\documentclass[12pt,amsymb,fullpage]{amsart}
\usepackage{amssymb,amscd,pstricks}

\newtheorem{theorem}{Theorem}[section]
\newtheorem{defn}[theorem]{Definition}

\newtheorem{lemma}[theorem]{Lemma}

\newtheorem{eple}[theorem]{Example}
\newtheorem{rmk}[theorem]{Remarks}
\newtheorem{dsc}[theorem]{Discussion}
\newtheorem{nota}[theorem]{Notation}

\newsavebox{\indbin}
\savebox{\indbin}{\begin{picture}(0,0)
\newlength{\gnu}
\settowidth{\gnu}{$\smile$} \setlength{\unitlength}{.5\gnu}
\put(-1,-.65){$\smile$} \put(-.25,.1){$|$}
\end{picture}}

\newcommand{\be}{\begin{enumerate}}
\newcommand{\bd}{\begin{defn}}
\newcommand{\bt}{\begin{theorem}}
\newcommand{\bl}{\begin{lemma}}
\newcommand{\ee}{\end{enumerate}}
\newcommand{\ed}{\end{defn}}
\newcommand{\et}{\end{theorem}}
\newcommand{\el}{\end{lemma}}

\begin{document}
\title{Some Arguments for the Wave Equation in Quantum Theory 3}
\author{Tristram de Piro}
\address{Flat 3, Redesdale House, 85 The Park, Cheltenham, GL50 2RP}
 \email{t.depiro@curvalinea.net}
\thanks{}
\begin{abstract}
We prove there exist solutions $(\rho,J))$ to the $1$-dimensional wave equation on $[-\pi,\pi]$, such that when $(\rho,J)$ are extended to a smooth solution $(\rho,\overline{J})$ of the continuity equation on a vanishing annulus $Ann(1,\epsilon)$ containing the unit circle $S^{1}$, with a corresponding causal solution $(\rho,\overline{J},\overline{E},\overline{B})$ to Maxwell's equations, obtained from Jefimenko's equations, the power radiated in a time cycle from any sphere $S(r)$ with $r>0$ is $O({1\over r})$, so that no power is radiated at infinity over a cycle.
\end{abstract}
\maketitle

\begin{section}{Introduction}
This paper is divided into three parts. In the first part, in Lemma \ref{wave}, we prove a simple result showing that we can produce a pair of solutions $(\Psi,J)$ to the $1$-dimensional wave equation on the circle, with given initial conditions $\Psi_{0}$ and ${\partial\Psi\over \partial t}|_{0}$ such that ${\partial \Psi \over \partial t}+{\partial J\over \partial x}=0$. This last equation allows us, in the second part, in Lemma \ref{bump}, to extend the pair $(\Psi,J)$ to a smooth pair $(\rho,\overline{J})$ satisfying the continuity equation ${\partial \rho\over \partial t}+\bigtriangledown\centerdot\overline{J}=0$ in three dimensions, which restricts to the pair $(\Psi,J)$ on the unit circle, but may not satisfy the three dimensional wave equation, see \cite{dep3}. We also prove some results about possible flows on the unit circle which satisfy the continuity equation. In the third part, we use a result from \cite{dep3} to construct a solution $(\rho,\overline{J},\overline{E},\overline{B})$ to Maxwell's equations from the pair $(\rho,\overline{J})$, using Jefimenko's equations, which is referred to as the causal solution in \cite{G}. In Lemma \ref{2terms}, we calculate the power radiated $P(r,t)$ over the sphere $S(r)$ of radius $r$ for a fundamental solution $\rho^{1}=cos(mx)cos(mt)$, $J^{1}=sin(mx)sin(mt)$ to the wave equation on the circle and note that it does not satisfy the no radiating condition, in the sense that $lim_{r\rightarrow\infty}P(r,t)=0$, even when averaged over a cycle $(t,t+{\pi\over m})$. The calculation does not involve any approximations. However, in Lemma \ref{infinity}, we find that we can satisfy the no radiation condition over a cycle, $lim_{r\rightarrow\infty}\int_{t}^{t+{\pi\over m}}P(r,t)dt=0$, by taking a linear combination of fundamental solutions, setting three of the parameters to be equal, and reversing the sign of the fourth. This result relies on Poynting's theorem and the fact that the mechanical energy of a charge and current configuration restricted to a circle, can be computed over any ball $B(r)$ with $r>1$. However, there are still some issues with computing the mechanical energy, which we hope can be resolved accurately in \cite{dep6}.\\
\indent The result is interesting because the Larmor formula predicts that the power radiated at infinity by an accelerating charge, $P={\mu_{0}q^{2}a^{2}\over 6\pi c}$ is non-zero, unless the particle travels in a straight line. However, for a collection of particles, we obtain an effective cancellation when averaged over a cycle. Rutherford also used Larmor's result to suggest that the electrons in an atomic orbit couldn't orbit the nucleus in circular or elliptical orbits, as they would lose energy and spiral into the nucleus. The result in the paper suggests, however, that there are atomic configurations of charge and current which retain energy over a cycle, and, therefore, wouldn't collapse as Rutherford predicted.

\indent In the final remark, we show that we can generalise the results of the paper to the $1$-dimensional wave equation with velocity $c$. Analogously, in the paper \cite{dep3}, we proved that for any initial conditions $\{\rho_{0},{\partial\rho\over \partial t}|_{0}\}$, there exists solutions $(\rho,\overline{J})$ to the three dimensional wave equations with velocity $c$, a connecting relation, with velocity $c$, and the continuity equation;\\

$\square^{2}\rho=0$, $\square^{2}\overline{J}=0$\\

$\bigtriangledown(\rho)+{1\over c^{2}}{\partial \overline{J}\over \partial t}=\overline{0}$ $(\dag)$\\

${\partial\rho\over\partial t}+\bigtriangledown\centerdot\overline{J}=0$\\

where $\square^{2}$ denotes the d'Alembertian operator, such that there \emph{exist} fields $\{\overline{E}_{S},\overline{B}_{S}\}$ in every inertial frame $S$, with $div(\overline{E}_{S}\times \overline{B}_{S})=0$, and $(\rho_{S},\overline{J}_{S},\overline{E}_{S},\overline{B}_{S})$ solutions to Maxwell's equation for the transformed charge and current $(\rho_{S},\overline{J}_{S})$. In particularly, the power radiated $P(r,t)$ over any sphere $S(r)$ is identically zero, without averaging, and irrespective of the inertial frame. This seems to be a stronger type of radiation than that considered in this paper, and would probably have to be generated over a sphere, using a cavity magnetron, rather than a circular antennae.

\indent In the paper \cite{dep7}, we proved that an atomic system which satisfies the no radiating condition and is in thermal equilibrium must also satisfy the equations $(\dag)$, but we have not yet shown the converse, that there is an atomic system which satisfies the equations $(\dag)$ and is in thermal equilibrium. However, the result of this paper together with ongoing work in \cite{dep5} and \cite{dep6}, yield a method for finding the appropriate initial conditions, see the final Lemma \ref{thermal}.

\end{section}
\begin{section}{The Wave Equation on a Circle}
\begin{defn}
\label{definitions}
We let $C^{\infty}([-\pi,\pi])$, $C(\mathcal{R})$ and $C^{\infty}(\mathcal{R})$ have their conventional meanings. We let $T=[-\pi,\pi]\times\mathcal{R}$ and let $T^{0}=(-\pi,\pi)\times\mathcal{R}$ denote its interior. We let $C(T)=\{G,\ continuous\ on\ T, G_{t}\in C([-\pi,\pi]),\ for\ t\in\mathcal{R}\}$, $\mathcal{S}(T)=\{G\in C(T):G_{t}\in C^{\infty}([-\pi,\pi])$, for $t\in\mathcal{R}, G|T^{0}\in C^{\infty}(T^{0})\}$

\end{defn}

\begin{defn}
\label{coeffs}
For any real $\{\Psi,\Psi_{0},\Psi_{1}\}\subset C^{\infty}([-\pi,\pi])$, we define the Fourier coefficients, for $m\in\mathcal{Z}$ by;\\

$\mathcal{F}(\Psi)(m)=\int_{1\over 2\pi}\int_{-\pi}^{\pi}\Psi_{0}(x)e^{-imx}dx$\\

$a_{m}={1\over 2\pi}\int_{-\pi}^{\pi}\Psi_{0}(x)cos(mx)dx$\\

$b_{m}={1\over 2\pi}\int_{-\pi}^{\pi}\Psi_{0}(x)sin(mx)dx$\\

$a'_{m}={1\over 2\pi}\int_{-\pi}^{\pi}\Psi_{1}(x)cos(mx)dx$\\

$b'_{m}={1\over 2\pi}\int_{-\pi}^{\pi}\Psi_{1}(x)sin(mx)dx$\\

\end{defn}

\begin{lemma}
\label{wave}
For any real $\{\Psi_{0},\Psi_{1}\}\subset C^{\infty}([-\pi,\pi])$, there exists a unique real $\Psi\in\mathcal{S}(T)$ solving the rescaled wave equation;\\

${\partial^{2} \Psi\over \partial t^{2}}-{\partial^{2}\Psi\over \partial x^{2}}=0$\\

with $\Psi(0,x)=\Psi_{0}(x)$, and ${\partial \Psi\over \partial t}(0,x)=\Psi_{1}(x)$ for $x\in [-\pi,\pi)$. Moreover, using the terminology of Definition \ref{coeffs}, $\Psi$ is given explicitly by the series;\\

$a_{0}+ta_{0}'+2\sum_{m\in\mathcal{Z}_{>0}}a_{m}cos(mx)cos(mt)+2\sum_{m\in\mathcal{Z}_{>0}}b_{m}sin(mx)cos(mt)$\\

$+2\sum_{m\in\mathcal{Z}_{>0}}{a'_{m}\over m}cos(mx)sin(mt)+2\sum_{m\in\mathcal{Z}_{>0}}{b'_{m}\over m}sin(mx)sin(mt)$\\

If $\Psi\in\mathcal{S}(T)$ denotes such a solution, with related charge density and current;\\

$\rho(x,t)=\Psi (x,t)$\\

$J(x,t)=\int_{-\pi}^{x}-{\partial \rho\over \partial t}dx$\\

Then $\{\rho, J\}$ satisfy the continuity equation;\\

${\partial \rho\over \partial t}+{\partial J\over \partial x}=0$\\

In particular if $P(t)=\int_{-\pi}^{\pi}\rho(x,t)dx$, then $P(t)=P(0)$ is constant. Moreover, when $\Psi_{0}$ and $\Psi_{1}$ are symmetric, we have the additional relation;\\

${\partial J\over \partial t}=-{\partial \rho\over \partial x}$\\

and $J$ also satisfies the wave equation with $J\in \mathcal{S}(T)$, and with $\{\rho,J\}$ given explicitly by;\\

$\rho(x,t)=a_{0}+a_{0}'t+2\sum_{m\in\mathcal{Z}>0}a_{m}cos(mx)cos(mt)+2\sum_{m\in\mathcal{Z}>0}{a'_{m}\over m}cos(mx)sin(mt)$\\

$J(x,t)=-a'_{0}\pi-a'_{0}x+2\sum_{m\in\mathcal{Z}>0}a_{m}sin(mx)sin(mt)-2\sum_{m\in\mathcal{Z}>0}{a'_{m}\over m}sin(mx)cos(mt)$\\

\end{lemma}
\begin{proof}
For the first part, suppose there exists $\Psi\in S(T)$ satisfying the hypotheses, then taking Fourier coefficients of the equation, for $m\in\mathcal{Z}$, and using integration by parts, we have that;\\

$\mathcal{F}({\partial^{2}\Psi\over \partial t^{2}}-{\partial^{2}\Psi\over \partial x^{2}})(m)$\\

$={d^{2}\mathcal{F}(\Psi)(m,t)\over dt^{2}}+m^{2}\mathcal{F}(\Psi)(m,t)=0$\\

Solving the resulting ODE's for $m\neq 0$, we have that;\\

$\mathcal{F}(\Psi)(m,t)=A_{m}e^{imt}+B_{m}e^{-imt}$\\

$=(A_{m}+B_{m})cos(mt)+i(A_{m}-B_{m})sin(mt)$\\

$\mathcal{F}(\Psi)(0,t)=A+Bt$\\

where $(A_{m}+B_{m})=\mathcal{F}(\Psi)(m,0)=\mathcal{F}(\Psi_{0})(m)$, $(imA_{m}-imB_{m})=\mathcal{F}({\partial \Psi\over \partial t})(m,0)=\mathcal{F}(\Psi_{1})(m)$, $A=\mathcal{F}(\Psi_{0})(0)$, $B=\mathcal{F}(\Psi_{1})(0)$. It follows that;\\

$\mathcal{F}(\Psi)(m,t)=\mathcal{F}(\Psi)(m,0)cos(mt)+{\mathcal{F}(\Psi_{1})(m)\over m}sin(mt)$, $(m\neq 0)$\\

$\mathcal{F}(\Psi)(0,t)=\mathcal{F}(\Psi_{0})(0)+\mathcal{F}(\Psi_{1})(0)t$, $(*)$\\

By the inversion theorem $\Psi$ is unique, and, using the fact that $\Psi$ is real, symmetry properties of the Fourier coefficients $\{a_{m},b_{m},a'_{m},b'_{m}\}$ and $(*)$;\\

$\Psi(x,t)=\sum_{m\in\mathcal{Z}}\mathcal{F}(\Psi)(m,t)e^{ixm}$\\

$=\mathcal{F}(\Psi_{0})(0)+\mathcal{F}(\Psi_{1})(0)t$\\

$+\sum_{m\in\mathcal{Z}_{\neq 0}}(\mathcal{F}(\Psi)(m,0)cos(mt)+{\mathcal{F}(\Psi_{1})(m)\over m}sin(mt))(cos(mx)+isin(mx))$\\

$=a_{0}+a'_{0}t+\sum_{m\in\mathcal{Z}_{\neq 0}}((a_{m}-ib_{m})cos(mt)(cos(mx)+isin(mx))$\\

$+\sum_{m\in\mathcal{Z}_{\neq 0}}({a'_{m}-ib'_{m}\over m})sin(mt))(cos(mx)+isin(mx))$\\

$=a_{0}+a'_{0}t+\sum_{m\in\mathcal{Z}_{\neq 0}}a_{m}cos(mx)cos(mt)+\sum_{m\in\mathcal{Z}_{\neq 0}}b_{m}sin(mx)cos(mt)$\\

$+\sum_{m\in\mathcal{Z}_{\neq 0}}{a'_{m}\over m}cos(mx)sin(mt)+\sum_{m\in\mathcal{Z}_{\neq 0}}{b'_{m}\over m}sin(mx)sin(mt)$\\

$=a_{0}+a'_{0}t+2\sum_{m\in\mathcal{Z}_{>0}}a_{m}cos(mx)cos(mt)+2\sum_{m\in\mathcal{Z}_{>0}}b_{m}sin(mx)cos(mt)$\\

$+2\sum_{m\in\mathcal{Z}_{>0}}{a'_{m}\over m}cos(mx)sin(mt)+2\sum_{m\in\mathcal{Z}_{>0}}{b'_{m}\over m}sin(mx)sin(mt)$\\

as required. It is easily checked that the above series also defines $\Psi\in S(T)$ with the required properties, settling the existence question.\\

For the second part, $\{\rho,J\}$ satisfy the continuity equation, by the definition of $J$ and the fundamental theorem of calculus. By inspection of the series for $\Psi$, it is clear that $J\in S(T)$. Moreover;\\

$P'(t)=\int_{-\pi}^{\pi}{\partial \rho\over \partial t}(x,t)dx$\\

$=\int_{-\pi}^{\pi}-{\partial J\over \partial x}(x,t)dx$\\

$=J(-\pi)-J(\pi)=0$\\

so that $P(t)=P(0)$ is constant. Differentiating under the integral sign, using the fact that $\rho$ satisfies the wave equation, and using the fundamental theorem of calculus again, we have;\\

${\partial J\over \partial t}=\int_{-\pi}^{x}-{\partial^{2}\rho\over \partial t^{2}}$\\

$=\int_{-\pi}^{x}-{\partial^{2}\rho\over \partial x^{2}}$\\

$=-{\partial\rho\over \partial x}+{\partial\rho\over \partial x}|_{-\pi}$\\

If $\Psi_{0}$ and $\Psi_{1}$ are symmetric, we have the above coefficients $\{b_{m},b_{m}'\}$ are zero, and ${\partial\rho\over \partial x}$ expands as a series in $sin(mx)$. It follows that;\\

${\partial\rho\over \partial x}|_{-\pi}=0$;\\

and;\\

${\partial J\over \partial t}=-{\partial\rho\over \partial x}$, so that, combined with the continuity equation, and the fact that the partial derivatives commutes, we obtain that;\\

${\partial^{2} J\over \partial t^{2}}=-{\partial^{2}\rho\over \partial x\partial t}=-{\partial^{2}\rho\over \partial t\partial x}={\partial^{2} J\over \partial x^{2}}$\\

Moreover, a simple calculation, using the definition of $J$ shows that;\\

$\rho(x,t)=a_{0}+a_{0}'t+2\sum_{m\in\mathcal{Z}>0}a_{m}cos(mx)cos(mt)+2\sum_{m\in\mathcal{Z}>0}{a'_{m}\over m}cos(mx)sin(mt)$\\

$J(x,t)=-a'_{0}\pi-a'_{0}x+2\sum_{m\in\mathcal{Z}>0}a_{m}sin(mx)sin(mt)-2\sum_{m\in\mathcal{Z}>0}{a'_{m}\over m}sin(mx)cos(mt)$\\

\end{proof}

\begin{rmk}
\label{bounded}
Observe that in the case of the wave equation, the solutions are bounded backwards in time, that is there exists a constant $C_{t}$, such that $|\Psi(x,t')|\leq C_{t}$, for all $t'\leq t$. This is an important consideration when we come to discuss the radiation condition behind Jefimenko's equations.
\end{rmk}
\end{section}

\begin{section}{Extending Charge and Current}
\begin{lemma}
\label{extension}
Let $\{\Psi,J\}$ be as in Lemma \ref{wave}, and $S(1)\subset z=0$ be the circle of radius $1$, centred at $(0,0,0)$ then, if $\rho$ is defined on $S(1)$ by;\\

$\rho(1,\theta,t)=\Psi(\theta,t)$, for $\theta\in [-\pi,\pi)$\\

and $\overline{J}$ is any smooth extension to the annulus $Ann(1,\epsilon)$, $0<\epsilon<1$, of $\overline{K}$, defined on $S(1)$ by;\\

$\overline{K}(1,\theta,t)=J(\theta,t)(-\sin(\theta),cos(\theta),0)$\\

then $\{\rho, \overline{J}\}$ satisfy the continuity equation;\\

${\partial \rho\over \partial t}+div(\overline{J})=0$\\

on $S(1)$.

\end{lemma}

\begin{proof}
Using Lemma \ref{wave}, it is sufficient to prove that, for $-\pi\leq \theta<\pi$, $t\in {\mathcal{R}}_{\geq 0}$;\\

$div(\overline{J})|_{(1,\theta,0,t)}=J'(\theta,t)$, $(*)$\\

Omitting the $t$ for ease of notation, and letting $\overline{J}=(J())$... we have that;\\

$div(-J(\theta)sin(\theta),J(\theta)cos(\theta),0)$\\

$={\partial \over \partial x}(-J(\theta)sin(\theta))+{\partial \over \partial y}(J(\theta)cos(\theta))$\\

We have;\\

${\partial \over \partial x}(-J(\theta)sin(\theta))$\\

$=-({\partial J\over \partial x}sin(\theta)+J(\theta){\partial sin(\theta)\over \partial x})$\\

$=-(J'(\theta){\partial \theta\over \partial x}sin(\theta)+J(\theta){\partial y\over \partial x})$, (as $r=1$ and $sin(\theta)={y\over r}=y$)\\

$=-(J'(\theta)sin(\theta){\partial \theta\over \partial x})$, as ${\partial y\over \partial x}=0$\\

$=-(J'(\theta)sin(\theta)-y)$\\

$=J'(\theta)sin(\theta)y$\\

(as $\theta=tan^{-1}({y\over x})$ and ${\partial \theta\over \partial x}={{-y\over x^{2}}\over 1+({y\over x})^{2}}={-y\over x^{2}+y^{2}}=-y$, with $r=1$)\\

Similarly;\\

${\partial \over \partial y}(J(\theta)cos(\theta))$\\

$=({\partial J\over \partial y}cos(\theta)+J(\theta){\partial cos(\theta)\over \partial y})$\\

$=(J'(\theta){\partial \theta\over \partial y}cos(\theta)+J(\theta){\partial x\over \partial y})$, (as $r=1$ and $cos(\theta)={x\over r}=x$)\\

$=(J'(\theta)cos(\theta){\partial \theta\over \partial y})$, as ${\partial x\over \partial y}=0$\\

$=(J'(\theta)cos(\theta)x)$\\

$=J'(\theta)cos(\theta)x$\\

(as $\theta=tan^{-1}({y\over x})$ and ${\partial \theta\over \partial y}={{1\over x}\over 1+({y\over x})^{2}}={x\over x^{2}+y^{2}}=x$, with $r=1$)\\

It follows that;\\

$div(-J(\theta)sin(\theta),J(\theta)cos(\theta),0)$\\

$=J'(\theta)sin(\theta)y+J'(\theta)cos(\theta)x$\\

$=J'(\theta)(sin^{2}(\theta)+cos^{2}(\theta))$\\

$=J'(\theta)$\\

(as $x=rcos(\theta)=cos(\theta)$, $x=rsin(\theta)=sin(\theta)$, when $r=1$)\\

\end{proof}
\begin{lemma}
\label{circularflow1}
Let $D(1)$ be the closed punctured disc, with radius $1$, and let $\rho$ on $D(1)$ be constant, then any smooth circular flow, with velocity $v(r)={w(r)\over \rho}\hat{\theta}$;\\

$\overline{J}(r,t)=w(r)(-sin(\theta),cos(\theta))$\\

satisfies the continuity equation.\\

Conversely, any smooth circular flow, $\{\rho,\overline{J}\}$, independent of time, satisfying the continuity equation, requires the density to depend only on $r$, with an equivalent flow $\{1,\overline{J}\}$, obtained with constant density $1$, by changing the velocity from $\overline{v}$ to $\rho \overline{v}$, where $\overline{J}=\rho\overline{v}$.\\

\end{lemma}
\begin{proof}

We have that ${\partial \rho\over \partial t}=0$, and;\\

$div(\overline{J})={\partial (-w(r)sin(\theta))\over \partial x}+{\partial (w(r)cos(\theta))\over \partial y}$\\

$=-{\partial w\over \partial x}sin(\theta)-w{\partial sin(\theta)\over \partial x}+{\partial w\over \partial y}cos(\theta)+w{\partial cos(\theta)\over \partial y}$\\

$=-w'(r){\partial r\over \partial x}sin(\theta)-w{\partial sin(\theta)\over \partial x}+w'(r){\partial r\over \partial y}cos(\theta)+w{\partial cos(\theta)\over \partial y}$\\

$=-w'(r){x\over r}sin(\theta)-{w(r)(-sin(\theta)cos(\theta))\over r}+w'(r){y\over r}cos(\theta)$\\

$-w(r){(sin(\theta)cos(\theta))\over r}$,(\footnote{ Using;\\

${\partial r\over \partial x}={{1\over 2}2x\over (x^{2}+y^{2})^{1\over 2}}={x\over (x^{2}+y^{2})^{1\over 2}}={x\over r}$\\

${\partial r\over \partial y}={{1\over 2}2y\over (x^{2}+y^{2})^{1\over 2}}={y\over (x^{2}+y^{2})^{1\over 2}}={y\over r}$\\

with $r=(x^{2}+y^{2})^{1\over 2}$.\\

and with $\theta=tan^{-1}({y\over x})$;\\

${\partial \theta \over \partial x}={-y\over x^{2}}{1\over 1+({y\over x})^{2}}={-y\over (x^{2}+y^{2})}={-y\over r^{2}}={-rsin(\theta)\over r^{2}}={-sin(\theta)\over r}$\\

${\partial \theta \over \partial y}={1\over x}{1\over 1+({y\over x})^{2}}={x\over (x^{2}+y^{2})}={x\over r^{2}}={rcos(\theta)\over r^{2}}={cos(\theta)\over r}$\\

${\partial sin(\theta)\over \partial x}=cos(\theta){\partial \theta\over \partial x}={-cos(\theta)sin(\theta)\over r}$\\

${\partial cos(\theta)\over \partial y}=-sin(\theta){\partial \theta\over \partial y}={-sin(\theta)cos(\theta)\over r}$\\

})\\

$=-w'(r)cos(\theta)sin(\theta)-{w(r)(-sin(\theta)cos(\theta))\over r}+w'(r)sin(\theta)cos(\theta)$\\

$-w(r){(sin(\theta)cos(\theta))\over r}=0$\\

with $x=r cos(\theta)$, $y=r sin(\theta)$\\

Conversely, suppose that $div(\overline{J})=0$, with;\\

$\overline{J}(r,\theta,t)=\rho(r,\theta)(-sin(\theta),cos(\theta))$\\

Then;\\

${\partial (-\rho(r,\theta) sin(\theta)) \over \partial x}+{\partial (\rho(r,\theta) cos(\theta)) \over \partial y}$\\

$=-{\partial \rho\over \partial x}sin(\theta)-\rho{\partial sin(\theta)\over \partial x}$\\

$+{\partial \rho\over \partial y}cos(\theta)+\rho{\partial cos(\theta)\over \partial y}$\\

$=-({\partial \rho\over \partial r}{\partial r\over \partial x}+{\partial \rho\over \partial \theta}{\partial \theta\over \partial x})sin(\theta)$\\

$-\rho({-sin(\theta)cos(\theta)\over r})$\\

$+({\partial \rho\over \partial r}{\partial r\over \partial y}+{\partial \rho\over \partial \theta}{\partial \theta\over \partial y})cos(\theta)$\\

$+\rho({-sin(\theta)cos(\theta)\over r})$\\

(using footnote 1, ${\partial \theta\over \partial x}={-sin(\theta)\over r}$, ${\partial \theta\over \partial y}={cos(\theta)\over r}$)\\

Therefore;\\

$div(\overline{J})=({\partial \rho\over \partial r}{\partial r\over \partial y}+{\partial \rho\over \partial \theta}{\partial \theta\over \partial y})cos(\theta)$\\

$-({\partial \rho\over \partial r}{\partial r\over \partial x}+{\partial \rho\over \partial \theta}{\partial \theta\over \partial x})sin(\theta)$\\

It follows that;\\

$div(\overline{J})$\\

$={\partial \rho\over \partial r}{y\over r}cos(\theta)+{\partial \rho\over \partial \theta}{\partial \theta\over \partial y}cos(\theta)$\\

$-{\partial \rho\over \partial r}{x\over r}sin(\theta)-{\partial \rho\over \partial \theta}{\partial \theta\over \partial x}sin(\theta)$\\

(with $y=rsin(\theta)$ and $x=rcos(\theta)$, using footnote 1 again; ${\partial r\over\partial x}={x\over r}=cos(\theta)$, ${\partial r\over\partial y}={y\over r}=sin(\theta)$)\\

Hence;\\

$div(\overline{J})$\\

$=({\partial \rho\over \partial r}sin(\theta)cos(\theta)+{\partial \rho\over \partial \theta}{cos^{2}(\theta)\over r})$\\

$-({\partial \rho\over \partial r}sin(\theta)cos(\theta)+{\partial \rho\over \partial \theta}{sin^{2}(\theta)\over r})$\\

$={\partial \rho\over \partial \theta}{1\over r}=0$\\

If $r\neq 0$, ${\partial \rho\over \partial \theta}=0$, so $\rho$ is independent of $\theta$.

\end{proof}
\begin{lemma}
\label{circularflow2}
Let $\{\Psi,J\}$ be as in Lemma \ref{extension}, with $\Psi$ non constant and independent of time, then any extension $\{\rho,\overline{J}\}$ of $\{\Psi,J\}$ to $Ann(1,\epsilon)$, $0<\epsilon<1$ which satisfies the continuity equation is not a circular flow.

\end{lemma}

\begin{proof}

By Lemma \ref{circularflow1}, any circular flow satisfying the continuity equation on $Ann(1,\epsilon)$ has a density $\rho$, depending only on $r$. In particular, $\rho|_{S(1)}$ is constant, contradicting the hypothesis.

\end{proof}

\begin{lemma}
\label{outwardandinwardflows1}{Determination of flows for density independent of time}

Suppose that $\rho(\theta,r)$ is smooth and independent of time on the annulus, defined by;\\

$Ann(1,\epsilon,\delta)=\{(\theta,r):-\pi\leq \theta<\pi, 1-\epsilon<r<1+\delta\}$\\

with smooth $\overline{J}(\theta,r)$, satisfying the continuity equation, and;\\

$\overline{J}(\theta,r)=(J_{1}(\theta,r),J_{2}(\theta,r))$\\

$=w_{1}(\theta,r)\hat{\overline{r}}+w_{2}(\theta,r)\hat{\overline{\theta}}$\\

Then for a given $\epsilon>0$, $1-\epsilon<r_{0}<1+\delta$, smooth boundary condition $g_{\epsilon}$ on $S(1-\epsilon)$, and smooth $w_{2}$ on $Ann(1,\epsilon,\delta)$, we obtain that;\\

$r_{0} w_{1}(r_{0},\theta_{0})={(1-\epsilon)g_{1-\epsilon}(\theta_{0})\over r_{0}}+{1\over r_{0}}\int_{1-\epsilon}^{r_{0}}-{\partial w_{2}\over \partial \theta}^{\theta_{0}}dr$\\

and, for $\epsilon=\delta=0$, on the $D(0,1)$, smooth $w_{2}$ on $D(0,1)$, with;\\

 $lim_{r\rightarrow 0}-{\partial w_{2}\over \partial \theta}^{\theta_{1}}=lim_{r\rightarrow 0}-{\partial w_{2}\over \partial \theta}^{\theta_{2}}$\\

 for all $\{\theta_{1},\theta_{2}\}\subset [-\pi,\pi)$.  with we obtain that, for $r_{0}\neq 0$;\\

$w_{1}(r_{0},\theta_{0})={1\over r_{0}}\int_{0}^{r_{0}}-{\partial w_{2}\over \partial \theta}^{\theta_{0}}dr$\\

and $w_{1}(0,0)=-w_{2}(0,0)$\\

with $w_{1}$ continuous.

\end{lemma}

\begin{proof}
Suppose that $\rho(\theta,r)$ is smooth and independent of time on the annulus, defined by;\\

$Ann(1,\epsilon,\delta)=\{(\theta,r):-\pi\leq \theta<\pi, 1-\delta<r<1+\epsilon\}$\\

with smooth $\overline{J}(\theta,r)$, satisfying the continuity equation, and;\\

$\overline{J}(\theta,r)=(J_{1}(\theta,r),J_{2}(\theta,r))$\\

$=w_{1}(\theta,r)\hat{\overline{r}}+w_{2}(\theta,r)\hat{\overline{\theta}}$\\

where $\hat{\overline{r}}=(cos(\theta),sin(\theta))$ and $\hat{\overline{\theta}}=(-sin(\theta),cos(\theta))$\\

It follows that;\\

$\overline{J}(\theta,r)$\\

$=w_{1}(\theta,r)(cos(\theta),sin(\theta))+w_{2}(\theta,r)(-sin(\theta),cos(\theta))$\\

$=(w_{1}cos(\theta)-w_{2}sin(\theta),w_{1}sin(\theta)+w_{2}cos(\theta))$\\

We have, using the hypotheses, that;\\

$div(\overline{J})={\partial\rho\over \partial t}=0$\\

We have that;\\

${\partial J_{1}\over \partial x}={\partial w_{1}\over \partial x}cos(\theta)+w_{1}{\partial(cos(\theta))\over \partial x}$\\

$-{\partial w_{2}\over \partial x}sin(\theta)-w_{2}{\partial(sin(\theta))\over \partial x}$\\

$={\partial w_{1}\over \partial x}cos(\theta)+w_{1}-sin(\theta){-y\over r^{2}}$\\

$-{\partial w_{2}\over \partial x}sin(\theta)-w_{2}cos(\theta){-y\over r^{2}}$\\

$={\partial w_{1}\over \partial x}cos(\theta)+w_{1}{sin(\theta)y\over r^{2}}$\\

$-{\partial w_{2}\over \partial x}sin(\theta)+w_{2}{cos(\theta)y\over r^{2}}$\\

${\partial J_{2}\over \partial y}={\partial w_{1}\over \partial y}sin(\theta)+w_{1}{\partial(sin(\theta))\over \partial y}$\\

$+{\partial w_{2}\over \partial y}cos(\theta)+w_{2}{\partial(cos(\theta))\over \partial y}$\\

$={\partial w_{1}\over \partial y}sin(\theta)+w_{1}cos(\theta){x\over r^{2}}$\\

$+{\partial w_{2}\over \partial y}cos(\theta)+w_{2}-sin(\theta){x\over r^{2}}$\\

$={\partial w_{1}\over \partial y}sin(\theta)+w_{1}cos(\theta){x\over r^{2}}$\\

$+{\partial w_{2}\over \partial y}cos(\theta)-w_{2}{sin(\theta)x\over r^{2}}$\\

Therefore;\\

$div(\overline{J})$\\

$=grad(w_{1})\centerdot \hat{\overline{r}}+{w_{1}\over r^{2}}(\overline{v_{\theta}}\centerdot flip(\overline{r}))$\\

$+grad(w_{2})\centerdot \hat{\overline{\theta}}+{w_{2}\over r^{2}}(\overline{w_{\theta}}\centerdot flip(\overline{r}))$, $(*)$\\

where $\overline{v_{\theta}}=(sin(\theta),cos(\theta))$, $\overline{w_{\theta}}=(cos(\theta),-sin(\theta))$, and $flip(\overline{r})=(y,x)$\\

We find $\{\alpha,\beta\}$ such that;\\

$\alpha\hat{\overline{\theta}}+\beta\hat{\overline{r}}=\overline{v_{\theta}}$\\

This is equivalent to finding $\{\alpha,\beta\}$ such that;\\

$M_{\theta}\overline{v_{\alpha,\beta}}=\overline{v_{\theta}}$\\

where $(M_{\theta})_{1,2}=(M_{\theta})_{2,1}=cos(\theta)$ and $(M_{\theta})_{1,1}=-(M_{\theta})_{2,2}=-sin(\theta)$. and  $\overline{v_{\alpha,\beta}}=(\alpha,\beta)$.\\

We have;\\

$\overline{v_{\alpha,\beta}}=M_{\theta}^{-1}\overline{v_{\theta}}$\\

$={1\over -sin^{2}(\theta)-cos^{2}(\theta)}N_{\theta}\overline{v_{\theta}}$\\

$=-(sin^{2}(\theta)-cos^{2}(\theta),-2sin(\theta)cos(\theta))$\\

$=(cos(2\theta),sin(2\theta))$\\

where;\\

$(N_{\theta})_{1,2}=(N_{\theta})_{2,1}=-cos(\theta)$ and $(N_{\theta})_{1,1}=-(N_{\theta})_{2,2}=sin(\theta)$.

It follows that $\alpha=cos(2\theta)$, $\beta=sin(2\theta)$. Similarly, we find $\{\alpha,\beta\}$ such that;\\

$\alpha\hat{\overline{\theta}}+\beta\hat{\overline{r}}=\overline{w_{\theta}}$\\

Again, this is equivalent to finding $\{\alpha,\beta\}$ such that;\\

$M_{\theta}\overline{v_{\alpha,\beta}}=\overline{w_{\theta}}$\\

We have;\\

$\overline{v_{\alpha,\beta}}=M_{\theta}^{-1}\overline{w_{\theta}}$\\

$={1\over -sin^{2}(\theta)-cos^{2}(\theta)}N_{\theta}\overline{w_{\theta}}$\\

$=-(2sin(\theta)cos(\theta),-cos^{2}(\theta)+sin^{2}(\theta))$\\

$=(-sin(2\theta),cos(2\theta))$\\

It follows that $\alpha=sin(2\theta)$, $\beta=cos(2\theta)$. Substituting in $(*)$, we obtain;\\

$div(\overline{J})$\\

$=grad(w_{1})\centerdot \hat{\overline{r}}+{w_{1}\over r^{2}}(cos(2\theta)\hat{\overline{\theta}}+sin(2\theta)\hat{\overline{r}})\centerdot flip(\overline{r})$\\

$+grad(w_{2})\centerdot \hat{\overline{\theta}}+{w_{2}\over r^{2}}(-sin(2\theta)\overline{\theta}+cos(2\theta){\overline{r}})\centerdot flip(\overline{r})$\\

$=(grad(w_{1})+{w_{1}\over r^{2}}flip(\overline{r})sin(2\theta)+{w_{2}\over r^{2}}flip(\overline{r})cos(2\theta))\centerdot \hat{\overline{r}}$\\

$+((grad(w_{2})-{w_{2}\over r^{2}}flip(\overline{r})sin(2\theta)+{w_{1}\over r^{2}}flip(\overline{r})cos(2\theta))\centerdot \hat{\overline{\theta}}$, $(**)$\\

We have that $\overline{r}=r\hat{\overline{r}}$ and determine $\{\alpha,\beta\}$ such that $flip(\overline{r})=\alpha{\overline{\theta}}+\beta\hat{\overline{r}}$\\

Similarly to the above, we obtain;\\

$\overline{v_{\alpha,\beta}}$\\

$=-(rsin^{2}(\theta)-rcos^{2}(\theta),-2rsin(\theta)cos(\theta))$\\

$=(rcos(2\theta),rsin(2\theta))$\\

$=(rcos(2\theta), rsin(2\theta))$\\

so that $\alpha=rcos(2\theta)$, $\beta=rsin(2\theta)$, and $flip(\overline{r})=rcos(2\theta)\hat{\overline{\theta}}+rsin(2\theta)\hat{\overline{r}}$\\

Substituting into $(**)$, and using the fact that $\{\hat{\overline{r}},\hat{\overline{\theta}}\}$ are orthonormal, we obtain;\\

$div(\overline{J})$\\

$=(grad(w_{1})+({w_{1}sin(2\theta)\over r^{2}}+{w_{2}cos(2\theta)\over r^{2}})(rcos(2\theta)\hat{\overline{\theta}}+rsin(2\theta){\overline{r}}))\centerdot\hat{\overline{r}}$\\

$+(grad(w_{2})+({-w_{2}sin(2\theta)\over r^{2}}+{w_{1}cos(2\theta)\over r^{2}})(rcos(2\theta)\hat{\overline{\theta}}+rsin(2\theta){\overline{r}}))\centerdot\hat{\overline{\theta}}$, $(***)$\\

$=grad(w_{1})\centerdot\hat{\overline{r}}+({(w_{1}sin(2\theta))rsin(2\theta)\over r^{2}}+{(w_{2}cos(2\theta))rsin(2\theta)\over r^{2}})+grad(w_{2})\centerdot\hat{\overline{\theta}}+({(w_{1}cos(2\theta))rcos(2\theta)\over r^{2}}-{(w_{2}sin(2\theta))rcos(2\theta)\over r^{2}})$\\

Writing;\\

$grad(w_{1})=\alpha\hat{\overline{r}}+\beta\hat{\overline{\theta}}$\\

$grad(w_{2})=\gamma\hat{\overline{r}}+\delta\hat{\overline{\theta}}$\\

with $\alpha=grad(w_{1})\centerdot\hat{\overline{r}}$ and $\delta=grad(w_{2})\centerdot\hat{\overline{\theta}}$, we obtain from $(***)$ and $div(\overline{J})=0$ that;\\

$\alpha+\delta={-w_{1}\over r}$\\

$={-w_{1}\over r}$, $(****)$\\

We have that;\\

$\alpha=grad(w_{1})\centerdot\hat{\overline{r}}$\\

$={\partial w_{1}\over \partial x}cos(\theta)+{\partial w_{1}\over \partial y}sin(\theta)$, \\

$\delta=grad(w_{2})\centerdot\hat{\overline{\theta}}$\\

$={\partial w_{2}\over \partial x}-sin(\theta)+{\partial w_{2}\over \partial y}cos(\theta)$ $(*****)$\\

We have, using the calculations for $\{{\partial \theta\over \partial x},{\partial\theta\over \partial y},{\partial r\over \partial x},{\partial r\over \partial y}\}$, from the previous lemma, and the chain rule;\\

${\partial w_{1}\over \partial x}={\partial w_{1}\over \partial \theta}{-sin(\theta)\over r}+{\partial w_{1}\over \partial r}cos(\theta)$\\

${\partial w_{1}\over \partial y}={\partial w_{1}\over \partial \theta}{cos(\theta)\over r}+{\partial w_{1}\over \partial r}sin(\theta)$\\

${\partial w_{2}\over \partial x}={\partial w_{2}\over \partial \theta}{-sin(\theta)\over r}+{\partial w_{2}\over \partial r}cos(\theta)$\\

${\partial w_{2}\over \partial y}={\partial w_{2}\over \partial \theta}{cos(\theta)\over r}+{\partial w_{2}\over \partial r}sin(\theta)$, $(******)$\\

It follows from $(*****)$ and $(******)$ that;\\

$\alpha=({\partial w_{1}\over \partial \theta}{-sin(\theta)\over r}+{\partial w_{1}\over \partial r}cos(\theta))cos(\theta)+({\partial w_{1}\over \partial \theta}{cos(\theta)\over r}+{\partial w_{1}\over \partial r}sin(\theta))sin(\theta)$\\

$\delta=({\partial w_{2}\over \partial \theta}{-sin(\theta)\over r}+{\partial w_{2}\over \partial r}cos(\theta))-sin(\theta)+({\partial w_{2}\over \partial \theta}{cos(\theta)\over r}+{\partial w_{2}\over \partial r}sin(\theta))cos(\theta)$\\

Simplifying and substituting into $(****)$, we obtain;\\

${\partial w_{1}\over \partial r}+{1\over r}{\partial w_{2}\over \partial \theta}={-w_{1}\over r}$\\

and rearranging;\\

$r{\partial w_{1}\over \partial r}+w_{1}=-{\partial w_{2}\over \partial \theta}$\\

Fixing $\theta_{0}$ and multiplying by $p(\theta_{0},r)$, we obtain;\\

$p(r,\theta_{0})r{dw_{1}^{\theta_{0}}\over dr}+p(r,\theta_{0})w_{1}^{\theta_{0}}=-p(r,\theta_{0}){\partial w_{2}\over \partial \theta}^{\theta_{0}}$\\

Letting $s(r,\theta_{0})=p(r,\theta_{0})r$, $(\dag)$\\

we have that;\\

$[s(r,\theta_{0})w_{1}(r,\theta_{0})]'$\\

$=s'(r,\theta_{0})w_{1}+sw_{1}'(r,\theta_{0})$\\

so, equating coefficients, we require;\\

$s'(r,\theta_{0})=p(r,\theta_{0})$, $(\dag\dag)$\\

Using $(\dag),(\dag\dag)$, and, assuming $p(r,\theta_{0})\neq 0$, we obtain;\\

${s'(r,\theta_{0})\over s(r,\theta_{0})}={1\over r}$\\

$ln(s(r,\theta_{0}))=ln(r)+d(\theta_{0})$\\

and, taking exponentials;\\

$s(r,\theta_{0})=A(\theta_{0})r$\\

where $A(\theta_{0})=e^{d(\theta_{0})}$\\

$[A(\theta_{0})r^{c(\theta_{0})}w_{1}(r,\theta_{0})]'=-p(r,\theta_{0}){\partial w_{2}\over \partial \theta}^{\theta_{0}}$\\

where $p(r,\theta_{0})=A(\theta_{0})$\\

Integrating, and using the fundamental theorem of calculus, we obtain that;\\

$[A(\theta_{0})r w_{1}(r,\theta_{0})]_{1-\epsilon}^{r_{0}}=\int_{1-\epsilon}^{r_{0}}-A(\theta_{0}){\partial w_{2}\over \partial \theta}^{\theta_{0}}dr$\\

Case 1; For a given $\epsilon,\delta>0$, and $1-\epsilon<r_{0}<1+\delta$, we obtain that;\\

$r_{0} w_{1}(r_{0},\theta_{0})-(1-\epsilon) w_{1}(1-\epsilon,\theta_{0})=\int_{1-\epsilon}^{r_{0}}-{\partial w_{2}\over \partial \theta}^{\theta_{0}}dr$\\

so free to choose a smooth $w_{2}$ on $Ann(\epsilon,\delta,1)$ and a smooth boundary condition for $w_{1}$ on $S^{1}(1-\epsilon)$, to obtain $w_{1}$.\\

Case 2; Letting $\epsilon=1,\delta=0$, and, assuming $w_{1}$ is defined at $(0,0)$, we must have that, for $r_{0}>0$;\\

$r_{0} w_{1}(r_{0},\theta_{0})=\int_{0}^{r_{0}}-{\partial w_{2}\over \partial \theta}^{\theta_{0}}dr$\\

By L'Hopital's rule, the fact that $r(0)=0$, $a(0)=0$, where $a(r)=\int_{0}^{r}-{\partial w_{2}\over \partial \theta}^{\theta_{0}}dr'$, and the Fundamental Theorem of Calculus, $a'(0)=-{\partial w_{2}\over \partial \theta}(0,0)$, we have that;\\

$lim_{r_{0}\rightarrow 0}{1\over r_{0}}\int_{0}^{r_{0}}-{\partial w_{2}\over \partial \theta}^{\theta_{0}}dr$\\

$={a'(0)\over r'(0)}$

$=-{\partial w_{2}\over \partial \theta}^{\theta_{0}}(0)$\\

so defining $w_{1}(0,0)=-w_{2}(0,0)$ generates a continuous solution, provided $lim_{r\rightarrow 0}-{\partial w_{2}\over \partial \theta}^{\theta_{1}}=lim_{r\rightarrow 0}-{\partial w_{2}\over \partial \theta}^{\theta_{2}}$, for all $\{\theta_{1},\theta_{2}\}\subset [-\pi,\pi)$, this is true if $w_{2}$ is analytic in $(x,y)$, (\footnote{In this case, we can write;\\

$w_{2,r}(x,y)=\sum_{m,n=0}^{\infty}w_{2}^{(m,n)}(0,0){x^{m}y^{n}\over m!n!}$\\

so that;\\

$w_{2}(r,\theta)=\sum_{m,n=0}^{\infty}w_{2}^{(m,n)}(0,0){{rcos(\theta)}^{m}{rsin(\theta)}^{n}\over m!n!}$\\

Then;\\

$lim_{r\rightarrow 0}({1\over r}\int_{0}^{r}w_{2}(r,\theta))$\\

$=lim_{r\rightarrow 0}({1\over r}\int_{0}^{r}\sum_{m,n=0}^{\infty}w_{2}^{(m,n)}(0,0){{rcos(\theta)}^{m}{rsin(\theta)}^{n}\over m!n!})$\\

$=\sum_{m,n=0}^{\infty}w_{2}^{(m,n)}(0,0)lim_{r\rightarrow 0}({1\over r}\int_{0}^{r}{{rcos(\theta)}^{m}{rsin(\theta)}^{n}\over m!n!})$\\

$=\sum_{m,n=0}^{\infty}w_{2}^{(m,n)}(0,0)({{rcos(\theta)}^{m}{rsin(\theta)}^{n}\over m!n!})|_{r=0}$\\

$=w_{2,r}(0,0)$\\

for all $\theta$, using the fundamental theorem of calculus again.\\

General case, need to check higher derivatives and use the product rule.

}).

\end{proof}

\begin{rmk}
It follows we are free to choose a smooth $w_{2}$ on $Ann(\epsilon,1)$ and a smooth boundary condition for $w_{1}$ on $S^{1}(1-\epsilon)$, to obtain $w_{1}$, so we are free to choose a smooth $w_{2}$ on $D(0,1)$, to obtain a smooth $w_{1}$.\\

\end{rmk}
\begin{lemma}
\label{bump}

Given $\{\Psi,J\}$ as in Lemma \ref{wave}, $0<\epsilon<1$ and notation as in Lemma \ref{extension}, there exists a smooth pair $(\rho_{1},\overline{J})$ supported on $Ann(1,\epsilon)\times (-\epsilon,\epsilon)$, such that $(\rho_{1},\overline{J})$ satisfies the continuity equation in $\mathcal{R}^{3}\times\mathcal{R}_{>0}$ and $\rho_{1}|_{S^{1}\times \{0\}}=\rho$, $\overline{J}|_{S^{1}\times \{0\}}=\overline{K}$.\\

\end{lemma}

\begin{proof}
Let $\Phi_{\epsilon}(r)=e^{-{1\over 1-{(r-1)^{2}\over \epsilon^{2}}}}$, if $r\in (1-\epsilon,1+\epsilon)$ and $\Phi_{\epsilon}(r)=0$ otherwise, $r>0$, then $\Phi_{\epsilon}$ is smooth on $\mathcal{R}_{>0}$ and supported on $(1-\epsilon,1+\epsilon)$. Let $\Phi_{1,\epsilon}(z)=e^{-{1\over 1-{z^{2}\over \epsilon^{2}}}}$, if $z\in (-\epsilon,\epsilon)$ and $\Phi_{1,\epsilon}(z)=0$ otherwise, $z\in\mathcal{R}$, then $\Phi_{1,\epsilon}$ is smooth on $\mathcal{R}$ and supported on $(-\epsilon,\epsilon)$. Define $(\rho_{1},\overline{J})$ on $\mathcal{R}^{2}\times\mathcal{R}_{>0}$ by;\\

$\rho_{1}(r,\theta,t)=\Phi_{\epsilon}(r)\rho(1,\theta,t)=\Phi_{\epsilon}(r)\Psi(\theta,t)$\\

for $r>0, \theta\in [-\pi,\pi),t>0$\\

$\rho_{1}(0,0,t)=0$, $t>0$\\

$\overline{J}(r,\theta,t)=\Phi_{\epsilon}(r)\overline{K}(1,\theta,t)=\Phi_{\epsilon}(r)J(\theta,t)(-sin(\theta,cos(\theta,0)$\\

for $r>0, \theta\in [-\pi,\pi),t>0$\\

$\overline{J}(0,0,t)=\overline{0}$, $t>0$\\

Then, using Lemma \ref{extension} and the facts that, for any given $r>0$, $\Phi_{\epsilon}(r)\Psi(\theta,t)$ and $\Phi_{\epsilon}(r)J(\theta,t)$ satisfy the conditions of Lemma \ref{wave}, we have that $(\rho_{1},\overline{J})$ satisfy the continuity equation on $\mathcal{R}^{2}\times\mathcal{R}_{>0}$, $(\dag)$. Now define $(\rho_{1},\overline{J})$ on $\mathcal{R}^{3}\times\mathcal{R}_{>0}$ by;\\

$\rho_{1}(r,\theta,z,t)=\Phi_{1,\epsilon}(z)\rho_{1}(r,\theta,t)$\\

for $r>0, \theta\in [-\pi,\pi), z\in\mathcal{R},t>0$\\

$\rho_{1}(0,0,z,t)=0$, $z\in \mathcal{R},t>0$\\

$\overline{J}(r,\theta,z,t)=\Phi_{1,\epsilon}(z)\overline{J}(r,\theta,t)$\\

for $r>0, \theta\in [-\pi,\pi), z\in\mathcal{R}, t>0$\\

$\overline{J}(0,0,z,t)=\overline{0}$, $z\in \mathcal{R},t>0$\\

Clearly, $(\rho_{1},\overline{J})$ is supported on $Ann(1,\epsilon)\times (-\epsilon,\epsilon)$, and if $\overline{J}=(j_{1},j_{2},j_{3})$, we have that $j_{3}(x,y,z,t)=0$, so that ${\partial j_{3}\over \partial z}=0$ and $\bigtriangledown\centerdot\overline{J}={\partial j_{1}\over \partial x}+{\partial j_{2}\over \partial y}$, $(\dag\dag)$. By $(\dag)$, we have that, for any $z\in\mathcal{R}$, $(\Phi_{1,\epsilon}(z)\rho_{1}(r,\theta,t),\Phi_{1,\epsilon}(z)\overline{J}(r,\theta,t))$ satisfies the continuity equation on $\mathcal{R}^{2}\times\mathcal{R}_{>0}$, so combining the result with $(\dag\dag)$, we obtain that $(\rho_{1},\overline{J})$ satisfies the continuity equation on $\mathcal{R}^{3}\times\mathcal{R}_{>0}$.\\

\end{proof}

\begin{rmk}
\label{analytic}

If we require the pair $(\rho_{1},\overline{J})$ to be real analytic, we can replace $\Phi_{\epsilon}(r)$ by $\Phi_{\epsilon,an}(r)=e^{-{(r-1)^{2}\over \epsilon^{2}}}$, if $r>0$, and $\Phi_{1,\epsilon}$ by $\Phi_{1,\epsilon,an}(z)=e^{-{z^{2}\over \epsilon^{2}}}$, if $z\in\mathcal{R}$, in the proof, leaving $(\rho_{1},\overline{J})$ to be undefined at $(0,0,z,t)$, for $z\in \mathcal{R}$, $t>0$.\\
\end{rmk}

\end{section}

\begin{section}{The No Radiation Condition}
We consider the charge density $\rho$ on $S(1)$, defined as in Lemma \ref{extension}, with corresponding current $\overline{K}$, which we also denote by $\overline{J}$, coming from the wave equation, so that $\{\rho, \overline{J}\}$ satisfy the continuity equation;\\

${\partial \rho\over \partial t}+div(\overline{J})=0$\\

on $S(1)$, for small extensions of $\{\rho, \overline{J}\}$.\\

In \cite{dep3}, it was shown that one then obtain an electromagnetic solution $(\rho,\overline{J},\overline{E},\overline{B})$, satisfying Maxwell's equations, given by the Jefimenko Equations;\\

$\overline{E}(\overline{r},t)={1\over 4\pi\epsilon_{0}}\int [{\rho(\overline{r'},t_{r})\over \mathfrak{r}^{2}}\hat{\mathfrak{\overline{r}}}+{\dot{\rho}(\overline{r'},t_{r})\over c\mathfrak{r}}\hat{\mathfrak{\overline{r}}}-{\dot{\overline{J}}(\overline{r'},t_{r})\over c^{2}\mathfrak{r}}]d\tau'$\\

$\overline{B}(\overline{r},t)={\mu_{0}\over 4\pi}\int [{\overline{J}(\overline{r'},t_{r})\over \mathfrak{r}^{2}}+{\dot{\overline{J}}(\overline{r'},t_{r})\over c\mathfrak{r}}]\times \hat{\mathfrak{\overline{r}}} d\tau'$, $(*)$\\

where $\mathfrak{r}=|\overline{r}-\overline{r'}|$, $\hat{\mathfrak{\overline{r}}}={\overline{r}-\overline{r'}\over |\overline{r}-\overline{r'}|}$, $t_{r}=t-{\mathfrak{r}\over c}$, $\tau'$ is the measure with respect to the integrand variable $\overline{r'}$.\\

We determine the conditions on $\{\rho, \overline{J}\}$ for which the no radiation condition holds, that is;\\

$lim_{r\rightarrow\infty}P(r)=0$\\

where $P(r)=\int_{S^{3}(r)}(\overline{E}\times \overline{B}).d\overline{S}(r)$ and $\{\overline{E}, \overline{B}\}$ are these (causal) fields, see \cite{G}.\\

\begin{lemma}
\label{magnetron}
Let $(\rho,\overline{J})$ be solutions to the three dimensional wave equations with velocity $c$, the connecting relation, with velocity $c$, and the continuity equation;\\

$\square^{2}\rho=0$, $\square^{2}\overline{J}=0$\\

$\bigtriangledown(\rho)+{1\over c^{2}}{\partial \overline{J}\over \partial t}=\overline{0}$ $(\dag)$\\

${\partial\rho\over\partial t}+\bigtriangledown\centerdot\overline{J}=0$\\

where $\square^{2}$ denotes the d'Alembertian operator. Then for any solutions $(\overline{E},\overline{B})$, such that $(\rho,\overline{J},\overline{E},\overline{B})$ satisfy Maxwell's equations, in particular for the causal solutions $(\overline{E},\overline{B})$ given by Jefimenko's equation, we have that;\\

$\square^{2}\overline{E}=\overline{0}$, $\square^{2}\overline{B}=\overline{0}$\\

and, the same result holds for the transformed current and charge, $(\rho_{S},\overline{J}_{S})$, in every inertial frame $S$.\\

\end{lemma}

\begin{proof}
By the proof in \cite{dep3}, we can find a pair $(\overline{E},\overline{B})$ in the base frame, with $\square^{2}\overline{E}=\overline{0}$,  and $\overline{B}=\overline{0}$, $(*)$, such that $(\rho,\overline{J},\overline{E},\overline{B})$ satisfy Maxwell's equations. If $(\overline{E}',\overline{B}')$ is any pair such that $(\rho,\overline{J},\overline{E}',\overline{B}')$ satisfy Maxwell's equations, then, taking the difference, $(0,\overline{0},\overline{E}-\overline{E}',\overline{B}-\overline{B}')$ is a vacuum solution to Maxwell's equations, so that, see \cite{G};\\

$\square^{2}(\overline{E}-\overline{E}')=\overline{0}$, $\square^{2}(\overline{B}-\overline{B}')=\overline{0}$\\

From $(*)$, we obtain that;\\

$\square^{2}\overline{E}'=\overline{0}$, $\square^{2}\overline{B}'=\overline{0}$\\

as well. The last claim follows from the above proof and the results in \cite{dep3}.\\

\end{proof}

\begin{rmk}
\label{light}
We conjecture that;\\

For any $\{\rho, \overline{J}\}$ satisfying the conditions from Lemma \ref{magnetron}, and corresponding causal $\{\overline{E}, \overline{B}\}$ from Jefimenko's equations, that $\{\overline{E}, \overline{B}\}$  satisfies the no radiation condition iff $\{\overline{E}+\overline{E}_{0},\overline{B}+\overline{B}_{0}\}$ satisfies the no radiation condition, for any corresponding pair $\{\overline{E}_{0},\overline{B}_{0}\}$, satisfying Maxwell's equations in vacuum.\\

In this case, as $\overline{B}=\overline{0}$, we obtain for the causal fields $(\overline{E}',\overline{B}')$ from Lemma \ref{magnetron}, that they satisfy the no radiation condition, and the same is true in every inertial frame $S$.

\end{rmk}

\begin{lemma}
\label{2terms}
Keeping the order in $(*)$ and writing;\\

$\overline{E}(\overline{r},t)=\overline{E}_{1}(\overline{r},t)+\overline{E}_{2}(\overline{r},t)+\overline{E}_{3}(\overline{r},t)$\\

$\overline{B}(\overline{r},t)=\overline{B}_{1}(\overline{r},t)+\overline{B}_{2}(\overline{r},t)$\\

we have that;\\

$lim_{r\rightarrow\infty}P(r)=lim_{r\rightarrow\infty}\int_{S(r)}(\overline{E}_{2}\times\overline{B}_{2}+\overline{E}_{3}\times\overline{B}_{2} ).d\overline{S}(r)$\\
\end{lemma}

\begin{proof}

We have that;\\

$\overline{E}\times \overline{B}=(\overline{E}_{1}+\overline{E}_{2}+\overline{E}_{3})\times (\overline{B}_{1}+\overline{B}_{2})$\\

A simple calculation shows that, as $c>1$, that for $|\overline{r}|>1$, $\overline{r}|\in S^{3}(r)$;\\

$|\overline{E}_{1}\times \overline{B}_{1}|_{\overline{r},t}\leq {\mu_{0}\over 16\pi^{2}\epsilon_{0}} \int_{S^{1}}\int_{S^{1}} {max_{\overline{r}',\overline{s}'\in S^{1}}(|p|(\overline{r'},t-t_{r})|\overline{J}|(\overline{s'},t-t_{r}))\over {(r-1)}^{4}}d\theta d\phi$\\

$\leq 4\pi^{2}{\mu_{0}\over 16\pi^{2}\epsilon_{0}}{max_{\overline{r}',\overline{s}'\in S^{1}}(|p|(\overline{r'},t-t_{r})|\overline{J}|(\overline{s'},t-t_{r}))\over {(r-1)}^{4}}$\\

$|\overline{E}_{1}\times \overline{B}_{2}|\leq {\mu_{0}\over 16\pi^{2}\epsilon_{0}} {max_{\overline{r}',\overline{s}'\in S^{1}}(|p|(\overline{r'},t-t_{r})|\dot{\overline{J}}|(\overline{s'},t-t_{r}))\over {(r-1)}^{3}}d\theta d\phi$\\

$\leq 4\pi^{2}{\mu_{0}\over 16\pi^{2}\epsilon_{0}}{max_{\overline{r}',\overline{s}'\in S^{1}}(|p|(\overline{r'},t-t_{r})|\dot{\overline{J}}|(\overline{s'},t-t_{r}))\over {(r-1)}^{3}}$\\

$|\overline{E}_{2}\times \overline{B}_{1}|\leq {\mu_{0}\over 16\pi^{2}\epsilon_{0}}{max_{\overline{r}',\overline{s}'\in S^{1}}(|\dot{p}|(\overline{r'},t-t_{r})\overline{J}|(\overline{s'},t-t_{r}))\over {(r-1)}^{3}}d\theta d\phi$\\

$\leq 4\pi^{2}{\mu_{0}\over 16\pi^{2}\epsilon_{0}}{max_{\overline{r}',\overline{s}'\in S^{1}}(|\dot{p}|(\overline{r'},t-t_{r})|\overline{J}|(\overline{s'},t-t_{r}))\over {(r-1)}^{3}}$\\

$|\overline{E}_{3}\times \overline{B}_{1}|\leq {\mu_{0}\over 16\pi^{2}\epsilon_{0}}{max_{\overline{r}',\overline{s}'\in S^{1}}(|\overline{J}|(\overline{r'},t-t_{r})|\dot{\overline{J}}|(\overline{s'},t-t_{r}))\over {(r-1)}^{3}}d\theta d\phi$\\

$\leq 4\pi^{2}{\mu_{0}\over 16\pi^{2}\epsilon_{0}}{max_{\overline{r}',\overline{s}'\in S^{1}}(|\overline{J}|(\overline{r'},t-t_{r})|\dot{\overline{J}}|(\overline{s'},t-t_{r}))\over {(r-1)}^{3}}$\\

It follows that, for $r>1$;\\

$max(|\int_{S^{3}(r)}(\overline{E}_{1}\times \overline{B}_{1}).d\overline{S}(r)|,\int_{S^{3}(r)}(\overline{E}_{1}\times \overline{B}_{2}).d\overline{S}(r)|,\int_{S^{3}(r)}(\overline{E}_{2}\times \overline{B}_{1}).d\overline{S}(r)|,\int_{S^{3}(r)}(\overline{E}_{3}\times \overline{B}_{1}).d\overline{S}(r)|)$\\

$=max(|\int_{S^{3}(r)}(\overline{E}_{1}\times \overline{B}_{1}).\hat{\overline{n}}dS(r)|,|\int_{S^{3}(r)}(\overline{E}_{1}\times \overline{B}_{2}).\hat{\overline{n}}dS(r)|,|\int_{S^{3}(r)}(\overline{E}_{2}\times \overline{B}_{1}).\hat{\overline{n}}dS(r)|,|\int_{S^{3}(r)}(\overline{E}_{3}\times \overline{B}_{1}).\hat{\overline{n}}dS(r)|)$\\

$\leq max(\int_{S^{3}(r)}|\overline{E}_{1}\times \overline{B}_{1})|dS(r)|,\int_{S^{3}(r)}|(\overline{E}_{1}\times \overline{B}_{2})|dS(r),\int_{S^{3}(r)}|(\overline{E}_{2}\times \overline{B}_{1})|dS(r),\int_{S^{3}(r)}|(\overline{E}_{3}\times \overline{B}_{1})|dS(r)|)$\\

$\leq \int_{S^{3}(r)}(max(|\overline{E}_{1}\times \overline{B}_{1}|(\overline{r},t),|\overline{E}_{1}\times \overline{B}_{2}|(\overline{r},t),|\overline{E}_{2}\times \overline{B}_{1}|(\overline{r},t),|\overline{E}_{3}\times \overline{B}_{1}|(\overline{r},t)))dS(r)$\\

$\leq 4\pi^{2} {Area(\delta(S^{3}(r)))\over (r-1)^{3}}{\mu_{0}\over 16\pi^{2}\epsilon_{0}}max_{\overline{r}',\overline{s}'\in S^{1}}(|p|(\overline{r}',t-t_{r})|\overline{J}(\overline{s}',t-t_{r})|,|p|(\overline{r}',t-t_{r})|\dot{\overline{J}}(\overline{s}',t-t_{r})|,|\dot{p}(\overline{r}',t-t_{r})||\overline{J}(\overline{s}',t-t_{r})|,|\overline{J}(\overline{r}',t-t_{r})||\dot{\overline{J}}((\overline{s}',t-t_{r}))|)$\\

$=4\pi^{2}{4\pi r^{2}\over (r-1)^{3}}{\mu_{0}\over 16\pi^{2}\epsilon_{0}}max_{\overline{r}',\overline{s}'\in S^{1}}(|p|(\overline{r}',t-t_{r})|\overline{J}(\overline{s}',t-t_{r})|,|p|(\overline{r}',t-t_{r})|\dot{\overline{J}}(\overline{s}',t-t_{r})|,|\dot{p}(\overline{r}',t-t_{r})||\overline{J}(\overline{s}',t-t_{r})|,|\overline{J}(\overline{r}',t-t_{r})||\dot{\overline{J}}((\overline{s}',t-t_{r}))|)$\\

$=C(r,t-t_{r})$, with $lim_{r\rightarrow\infty}C(r,t-t_{r})=0$, for $t\in\mathcal{R}$, given the assumption that $lim_{t\rightarrow -\infty}max_{x\in S^{1}}(|p|,|\dot{p}|,|\overline{J}|,|\dot{\overline{J}})|_{x,t}\leq D$, with $D\in\mathcal{R}$, see remark \ref{bounded}.

It follows that, for $t\in\mathcal{R}$\\

$lim_{r\rightarrow\infty}P(r,t)$\\

$=lim_{r\rightarrow\infty}\int_{S^{3}(r)}(\overline{E}\times \overline{B})(\overline{r},t).d\overline{S}(r)$\\

$=lim_{r\rightarrow\infty}\int_{S^{3}(r)}(\overline{E}_{2}\times \overline{B}_{2}+\overline{E}_{3}\times \overline{B}_{2})(\overline{r},t).d\overline{S}(r)$\\

as required.\\
\end{proof}

\begin{lemma}
\label{waveblock}
Let $\rho(x,t)=cos(mx)cos(mt)$, $J(x,t)=sin(mx)sin(mt)$, for $m\in\mathcal{Z}>0$, $x\in [-\pi,\pi)$, $t\in\mathcal{R}$, with corresponding $\rho(\theta,t)$ and;\\

$\overline{J}(\theta,t)=J(\theta)(-sin(\theta),cos(\theta),0)$\\

with $\theta\in [-\pi,\pi)$. Let $\overline{E}_{m}$ and $\overline{B}_{m}$ be the causal fields determined by Jefimenko's equations, then, if $m$ is even;\\

$P(r,t)=\beta\gamma m^{2}[-sin^{2}(mt)cos^{2}(m{(r^{2}+1)^{1\over 2}\over c})-cos^{2}(mt)sin^{2}(m{(r^{2}+1)^{1\over 2}\over c})$\\

$+2sin(mt)cos(mt)sin(m{(r^{2}+1)^{1\over 2}\over c})cos(m{(r^{2}+1)^{1\over 2}\over c})]$\\

$\sum_{w=0}^{\infty}\sum_{w'=0}^{\infty}\sum_{s\leq m,s,odd,s'\leq m,s',odd}(-1)^{s-1\over 2}C^{m}_{s}{(-1)^{w}m^{2w+1}\over (2w+1)!(r^{2}+1)^{w+{1\over 2}}}$\\

$(-1)^{s'-1\over 2}C^{m}_{s'}{(-1)^{w'}m^{2w'+1}\over (2w'+1)!(r^{2}+1)^{w'+{1\over 2}}}r^{2w+2w'+5}c_{w,w',s,s',m}+O({1\over r})$\\

and, if $m$ is odd;\\

$P(r,t)=\beta\gamma m^{2}[-cos^{2}(mt)cos^{2}(m{(r^{2}+1)^{1\over 2}\over c})-sin^{2}(mt)sin^{2}(m{(r^{2}+1)^{1\over 2}\over c})$\\

$-2sin(mt)cos(mt)sin(m{(r^{2}+1)^{1\over 2}\over c})cos(m{(r^{2}+1)^{1\over 2}\over c})]$\\

$\sum_{w=0}^{\infty}\sum_{w'=0}^{\infty}\sum_{s\leq m,s,odd,s'\leq m,s',odd}(-1)^{s-1\over 2}C^{m}_{s}{(-1)^{w}m^{2w}\over (2w)!(r^{2}+1)^{w}}(-1)^{s'-1\over 2}$\\

$C^{m}_{s'}{(-1)^{w'}m^{2w'}\over (2w')!(r^{2}+1)^{w'}}r^{2w+2w'+3}d_{w,w',s,s',m}+O({1\over r})$\\

\end{lemma}

\begin{proof}
We have that;\\

$\overline{E}_{2}(\overline{r},t)={1\over 4\pi\epsilon_{0}}\int [{\dot{\rho}(\overline{r'},t_{r})\over c\mathfrak{r}}\hat{\mathfrak{\overline{r}}}]d\tau'$\\

where;\\

 $\overline{r}=(x,y,z)$, $\overline{r'}=(cos(\theta,sin(\theta),0)$\\

 $\overline{r}-\overline{r'}=(x-cos(\theta),y-sin(\theta),z)$\\

 $\mathfrak{r}=|\overline{r}-\overline{r'}|=(x^{2}+y^{2}+z^{2}+1-2xcos(\theta)-2ysin(\theta))^{1\over 2}$\\

 $=(r^{2}+1-2xcos(\theta)-2ysin(\theta))^{1\over 2}$\\

 $\hat{\overline{\mathfrak{r}}}={(\overline{r}-\overline{r'})\over \mathfrak{r}}={(x-cos(\theta),y-sin(\theta),z)\over (x^{2}+y^{2}+z^{2}+1-2xcos(\theta)-2ysin(\theta))^{1\over 2}}$\\

 $\dot{\rho}(\theta,t)=-mcos(m\theta)sin(mt)$\\

 $\overline{E}_{2}(\overline{r},t)={1\over 4\pi\epsilon_{0}c}\int_{-\pi}^{\pi}{-mcos(m\theta)sin(mt_{r})\over (r^{2}+1-2xcos(\theta)-2ysin(\theta))}(x-cos(\theta),y-sin(\theta),z)d\theta$\\

$={1\over 4\pi\epsilon_{0}c(r^{2}+1)}\int_{-\pi}^{\pi}{-mcos(m\theta)sin(m(t-{\mathfrak{r}\over c}))\over (1-{2xcos(\theta)-2ysin(\theta)\over (r^{2}+1)})}(x-cos(\theta),y-sin(\theta),z)d\theta$\\

$={1\over 4\pi\epsilon_{0}c(r^{2}+1)}\int_{-\pi}^{\pi}-mcos(m\theta)sin(m(t-{\mathfrak{r}\over c}))(x-cos(\theta),y-sin(\theta),z)d\theta+O({1\over r^{2}})$\\

$={1\over 4\pi\epsilon_{0}c(r^{2}+1)}\int_{-\pi}^{\pi}-mcos(m\theta)[sin(mt)cos(m{\mathfrak{r}\over c})-cos(mt)sin(m{\mathfrak{r}\over c})](x,y,z)d\theta+O({1\over r^{2}})$\\

$={-msin(mt)\over 4\pi\epsilon_{0}c(r^{2}+1)}\int_{-\pi}^{\pi}cos(m\theta)cos(m{\mathfrak{r}\over c})(x,y,z)d\theta$\\

$+{mcos(mt)\over 4\pi\epsilon_{0}c(r^{2}+1)}\int_{-\pi}^{\pi}cos(m\theta)sin(m{\mathfrak{r}\over c})(x,y,z)d\theta+O({1\over r^{2}})$\\

Observe that, using Taylor expansions;\\

$cos(m{\mathfrak{r}\over c})=cos(m{(r^{2}+1)^{1\over 2}\over c}(1+{xcos(\theta)+ysin(\theta)\over r^{2}+1}))+O({1\over r})$\\

$=cos(m{(r^{2}+1)^{1\over 2}\over c})cos({mxcos(\theta)+mysin(\theta)\over (r^{2}+1)^{1\over 2}})-sin(m{(r^{2}+1)^{1\over 2}\over c})sin({mxcos(\theta)+mysin(\theta)\over (r^{2}+1)^{1\over 2}})+O({1\over r})$

$sin(m{\mathfrak{r}\over c})=sin(m{(r^{2}+1)^{1\over 2}\over c}(1+{xcos(\theta)+ysin(\theta)\over r^{2}+1}))+O({1\over r})$\\

$=sin(m{(r^{2}+1)^{1\over 2}\over c})cos({mxcos(\theta)+mysin(\theta)\over (r^{2}+1)^{1\over 2}})+cos(m{(r^{2}+1)^{1\over 2}\over c})sin({mxcos(\theta)+mysin(\theta)\over (r^{2}+1)^{1\over 2}})+O({1\over r})$\\

so that;\\

$\overline{E}_{2}(\overline{r},t)={-msin(mt)cos(m{(r^{2}+1)^{1\over 2}\over c})\over 4\pi\epsilon_{0}c(r^{2}+1)}\int_{-\pi}^{\pi}cos(m\theta)cos({mxcos(\theta)+mysin(\theta)\over (r^{2}+1)^{1\over 2}})(x,y,z)d\theta$\\

$+{msin(mt)sin(m{(r^{2}+1)^{1\over 2}\over c})\over 4\pi\epsilon_{0}c(r^{2}+1)}\int_{-\pi}^{\pi}cos(m\theta)sin({mxcos(\theta)+mysin(\theta)\over (r^{2}+1)^{1\over 2}})(x,y,z)d\theta$\\

$+{mcos(mt)sin(m{(r^{2}+1)^{1\over 2}\over c})\over 4\pi\epsilon_{0}c(r^{2}+1)}\int_{-\pi}^{\pi}cos(m\theta)cos({mxcos(\theta)+mysin(\theta)\over (r^{2}+1)^{1\over 2}})(x,y,z)d\theta$\\

$+{mcos(mt)cos(m{(r^{2}+1)^{1\over 2}\over c})\over 4\pi\epsilon_{0}c(r^{2}+1)}\int_{-\pi}^{\pi}cos(m\theta)sin({mxcos(\theta)+mysin(\theta)\over (r^{2}+1)^{1\over 2}})(x,y,z)d\theta+O({1\over r^{2}})$ $(*)$\\

We have that;\\

$cos(m\theta)=Re((cos(\theta)+isin(\theta))^{m})$\\

$=\sum_{s\leq m,s, even}(-1)^{s\over 2}C^{m}_{s}cos^{m-s}(\theta)sin^{s}(\theta)$\\

$cos({mxcos(\theta)+mysin(\theta)\over (r^{2}+1)^{1\over 2}})$\\

$=\sum_{w=0}^{\infty}{(-1)^{w}m^{2w}(xcos(\theta)+ysin(\theta))^{2w}\over (2w)!(r^{2}+1)^{w}}$\\

$=\sum_{w=0}^{\infty}{(-1)^{w}m^{2w}\over (2w)!(r^{2}+1)^{w}}\sum_{v=0}^{2w}C^{2w}_{v}x^{2w-v}cos^{2w-v}(\theta)y^{v}sin^{v}(\theta)$\\

$sin({mxcos(\theta)+mysin(\theta)\over (r^{2}+1)^{1\over 2}})$\\

$=\sum_{w=0}^{\infty}{(-1)^{w}m^{2w+1}(xcos(\theta)+ysin(\theta))^{2w+1}\over (2w+1)!(r^{2}+1)^{w+{1\over 2}}}$\\

$=\sum_{w=0}^{\infty}{(-1)^{w}m^{2w+1}\over (2w+1)!(r^{2}+1)^{w+{1\over 2}}}\sum_{v=0}^{2w+1}C^{2w+1}_{v}x^{2w+1-v}cos^{2w+1-v}(\theta)y^{v}sin^{v}(\theta)$\\

so that;\\

$\overline{E}_{2}(\overline{r},t)={-msin(mt)cos(m{(r^{2}+1)^{1\over 2}\over c})\over 4\pi\epsilon_{0}c(r^{2}+1)}\sum_{w=0}^{\infty}\sum_{s\leq m,s, even}(-1)^{s\over 2}C^{m}_{s}{(-1)^{w}m^{2w}\over (2w)!(r^{2}+1)^{w}}\sum_{v=0}^{2w}C^{2w}_{v}x^{2w-v}y^{v}(x,y,z)$\\

$\int_{-\pi}^{\pi}cos^{m-s}(\theta)sin^{s}(\theta)cos^{2w-v}(\theta)sin^{v}(\theta)d\theta$\\

$+{msin(mt)sin(m{(r^{2}+1)^{1\over 2}\over c})\over 4\pi\epsilon_{0}c(r^{2}+1)}\sum_{w=0}^{\infty}\sum_{s\leq m,s, even}(-1)^{s\over 2}C^{m}_{s}{(-1)^{w}m^{2w+1}\over (2w+1)!(r^{2}+1)^{w+{1\over 2}}}\sum_{v=0}^{2w+1}C^{2w+1}_{v}x^{2w+1-v}y^{v}(x,y,z)$\\

$\int_{-\pi}^{\pi}cos^{m-s}(\theta)sin^{s}(\theta)cos^{2w+1-v}(\theta)sin^{v}(\theta)d\theta$\\

$+{mcos(mt)sin(m{(r^{2}+1)^{1\over 2}\over c})\over 4\pi\epsilon_{0}c(r^{2}+1)}\sum_{w=0}^{\infty}\sum_{s\leq m,s, even}(-1)^{s\over 2}C^{m}_{s}{(-1)^{w}m^{2w}\over (2w)!(r^{2}+1)^{w}}\sum_{v=0}^{2w}C^{2w}_{v}x^{2w-v}y^{v}(x,y,z)$\\

$\int_{-\pi}^{\pi}cos^{m-s}(\theta)sin^{s}(\theta)cos^{2w-v}(\theta)sin^{v}(\theta)d\theta$\\

$+{mcos(mt)cos(m{(r^{2}+1)^{1\over 2}\over c})\over 4\pi\epsilon_{0}c(r^{2}+1)}\sum_{w=0}^{\infty}\sum_{s\leq m,s, even}(-1)^{s\over 2}C^{m}_{s}{(-1)^{w}m^{2w+1}\over (2w+1)!(r^{2}+1)^{w+{1\over 2}}}\sum_{v=0}^{2w+1}C^{2w+1}_{v}x^{2w+1-v}y^{v}(x,y,z)$\\

$\int_{-\pi}^{\pi}cos^{m-s}(\theta)sin^{s}(\theta)cos^{2w+1-v}(\theta)sin^{v}(\theta)d\theta+O({1\over r^{2}})$\\

We recall the result that;\\

$I_{\alpha}^{\beta}=\int_{-\pi}^{\pi}cos^{\alpha}(\theta)sin^{\beta}(\theta)d\theta={\pi\alpha!\beta!\over 2^{\alpha+\beta-1}({\alpha\over 2})!({\beta\over 2})!({\alpha+\beta\over 2})!}$\\

for $\alpha\geq 0$, $\beta\geq 0$, $\alpha$ and $\beta$ even.\\

$I_{\alpha}^{\beta}=0$ otherwise\\

Applying the result in this case, we obtain;\\

if $m$ is even, then;\\

$\overline{E}_{2}(\overline{r},t)={-msin(mt)cos(m{(r^{2}+1)^{1\over 2}\over c})\over 4\pi\epsilon_{0}c(r^{2}+1)}\sum_{w=0}^{\infty}\sum_{s\leq m,s, even}(-1)^{s\over 2}C^{m}_{s}{(-1)^{w}m^{2w}\over (2w)!(r^{2}+1)^{w}}$\\

$\sum_{v=0,v,even}^{2w}C^{2w}_{v}x^{2w-v}y^{v}(x,y,z)I_{2w+m-s-v}^{s+v}$\\

$+{mcos(mt)sin(m{(r^{2}+1)^{1\over 2}\over c})\over 4\pi\epsilon_{0}c(r^{2}+1)}\sum_{w=0}^{\infty}\sum_{s\leq m,s, even}(-1)^{s\over 2}C^{m}_{s}{(-1)^{w}m^{2w}\over (2w)!(r^{2}+1)^{w}}$\\

$\sum_{v=0,v,even}^{2w}C^{2w}_{v}x^{2w-v}y^{v}(x,y,z)I_{2w+m-s-v}^{s+v}+O({1\over r^{2}})$ $(**)$\\

if $m$ is odd, then;\\

$\overline{E}_{2}(\overline{r},t)={msin(mt)sin(m{(r^{2}+1)^{1\over 2}\over c})\over 4\pi\epsilon_{0}c(r^{2}+1)}\sum_{w=0}^{\infty}\sum_{s\leq m,s, even}(-1)^{s\over 2}C^{m}_{s}{(-1)^{w}m^{2w+1}\over (2w+1)!(r^{2}+1)^{w+{1\over 2}}}$\\

$\sum_{v=0,v,even}^{2w+1}C^{2w+1}_{v}x^{2w+1-v}y^{v}(x,y,z)I_{2w+1+m-s-v}^{s+v}$\\

$+{mcos(mt)cos(m{(r^{2}+1)^{1\over 2}\over c})\over 4\pi\epsilon_{0}c(r^{2}+1)}\sum_{w=0}^{\infty}\sum_{s\leq m,s, even}(-1)^{s\over 2}C^{m}_{s}{(-1)^{w}m^{2w+1}\over (2w+1)!(r^{2}+1)^{w+{1\over 2}}}$\\

$\sum_{v=0,v, even}^{2w+1}C^{2w+1}_{v}x^{2w+1-v}y^{v}(x,y,z)I_{2w+1+m-s-v}^{s+v}+O({1\over r^{2}})$ $(***)$\\

We have that;\\

$\overline{E}_{3}(\overline{r},t)={-1\over 4\pi\epsilon_{0}}\int [{\dot{\overline{J}}(\overline{r'},t_{r})\over c^{2}\mathfrak{r}}]d\tau'$\\

$\dot{\overline{J}}(\theta,t)=msin(m\theta)cos(mt)(-sin(\theta),cos(\theta),0)$\\

$\overline{E}_{3}(\overline{r},t)={-1\over 4\pi\epsilon_{0}c^{2}}\int_{-\pi}^{\pi}{msin(m\theta)cos(mt_{r})\over (r^{2}+1-2xcos(\theta)-2ysin(\theta))^{1\over 2}}(-sin(\theta),cos(\theta),0)d\theta$\\

$={-1\over 4\pi\epsilon_{0}c^{2}(r^{2}+1)^{1\over 2}}\int_{-\pi}^{\pi}{msin(m\theta)cos(m(t-{\mathfrak{r}\over c}))\over (1-{2xcos(\theta)-2ysin(\theta)\over (r^{2}+1)})^{1\over 2}}(-sin(\theta),cos(\theta),0)d\theta$\\

$={-1\over 4\pi\epsilon_{0}c^{2}(r^{2}+1)^{1\over 2}}\int_{-\pi}^{\pi}msin(m\theta)cos(m(t-{\mathfrak{r}\over c}))(-sin(\theta),cos(\theta),0)d\theta+O({1\over r^{2}})$\\

$={-1\over 4\pi\epsilon_{0}c^{2}(r^{2}+1)^{1\over 2}}\int_{-\pi}^{\pi}msin(m\theta)[cos(mt)cos(m{\mathfrak{r}\over c})+sin(mt)sin(m{\mathfrak{r}\over c})](-sin(\theta),cos(\theta),0)d\theta+O({1\over r^{2}})$\\

$={-mcos(mt)\over 4\pi\epsilon_{0}c^{2}(r^{2}+1)^{1\over 2}}\int_{-\pi}^{\pi}sin(m\theta)cos(m{\mathfrak{r}\over c})(-sin(\theta),cos(\theta),0)d\theta$\\

$-{msin(mt)\over 4\pi\epsilon_{0}c^{2}(r^{2}+1)^{1\over 2}}\int_{-\pi}^{\pi}sin(m\theta)sin(m{\mathfrak{r}\over c})(-sin(\theta),cos(\theta),0)d\theta+O({1\over r^{2}})$\\

$={-mcos(mt)cos(m(r^{2}+1)^{1\over 2})\over 4\pi\epsilon_{0}c^{2}(r^{2}+1)^{1\over 2}}\int_{-\pi}^{\pi}sin(m\theta)cos({mxcos(\theta)+mysin(\theta)\over (r^{2}+1)^{1\over 2}})(-sin(\theta),cos(\theta),0)d\theta$\\

$+{mcos(mt)sin(m(r^{2}+1)^{1\over 2})\over 4\pi\epsilon_{0}c^{2}(r^{2}+1)^{1\over 2}}\int_{-\pi}^{\pi}sin(m\theta)sin({mxcos(\theta)+mysin(\theta)\over (r^{2}+1)^{1\over 2}})(-sin(\theta),cos(\theta),0)d\theta$\\

$-{msin(mt)sin(m(r^{2}+1)^{1\over 2})\over 4\pi\epsilon_{0}c^{2}(r^{2}+1)^{1\over 2}}\int_{-\pi}^{\pi}sin(m\theta)cos({mxcos(\theta)+mysin(\theta)\over (r^{2}+1)^{1\over 2}})(-sin(\theta),cos(\theta),0)d\theta$\\

$-{msin(mt)cos(m(r^{2}+1)^{1\over 2})\over 4\pi\epsilon_{0}c^{2}(r^{2}+1)^{1\over 2}}\int_{-\pi}^{\pi}sin(m\theta)sin({mxcos(\theta)+mysin(\theta)\over (r^{2}+1)^{1\over 2}})(-sin(\theta),cos(\theta),0)d\theta+O({1\over r^{2}})$ $(****)$\\

We have that;\\

$sin(m\theta)=Im((cos(\theta)+isin(\theta))^{m})$\\

$=\sum_{s\leq m,s odd}(-1)^{s-1\over 2}C^{m}_{s}cos^{m-s}(\theta)sin^{s}(\theta)$\\

so that;\\

$\overline{E}_{3}(\overline{r},t)={-mcos(mt)cos(m{(r^{2}+1)^{1\over 2}\over c})\over 4\pi\epsilon_{0}c^{2}(r^{2}+1)^{1\over 2}}\sum_{w=0}^{\infty}\sum_{s\leq m,s, odd}(-1)^{s-1\over 2}C^{m}_{s}{(-1)^{w}m^{2w}\over (2w)!(r^{2}+1)^{w}}$\\

$\sum_{v=0}^{2w}C^{2w}_{v}x^{2w-v}y^{v}\int_{-\pi}^{\pi}(-sin(\theta),cos(\theta),0)cos^{m-s}(\theta)sin^{s}(\theta)cos^{2w-v}(\theta)sin^{v}(\theta)d\theta$\\

$+{mcos(mt)sin(m{(r^{2}+1)^{1\over 2}\over c})\over 4\pi\epsilon_{0}c^{2}(r^{2}+1)^{1\over 2}}\sum_{w=0}^{\infty}\sum_{s\leq m,s, odd}(-1)^{s-1\over 2}C^{m}_{s}{(-1)^{w}m^{2w+1}\over (2w+1)!(r^{2}+1)^{w+{1\over 2}}}$\\

$\sum_{v=0}^{2w+1}C^{2w+1}_{v}x^{2w+1-v}y^{v}\int_{-\pi}^{\pi}(-sin(\theta),cos(\theta),0)cos^{m-s}(\theta)sin^{s}(\theta)cos^{2w+1-v}(\theta)sin^{v}(\theta)d\theta$\\

$-{msin(mt)sin(m{(r^{2}+1)^{1\over 2}\over c})\over 4\pi\epsilon_{0}c^{2}(r^{2}+1)^{1\over 2}}\sum_{w=0}^{\infty}\sum_{s\leq m,s, odd}(-1)^{s-1\over 2}C^{m}_{s}{(-1)^{w}m^{2w}\over (2w)!(r^{2}+1)^{w}}$\\

$\sum_{v=0}^{2w}C^{2w}_{v}x^{2w-v}y^{v}\int_{-\pi}^{\pi}(-sin(\theta),cos(\theta),0)cos^{m-s}(\theta)sin^{s}(\theta)cos^{2w-v}(\theta)sin^{v}(\theta)d\theta$\\

$-{msin(mt)cos(m{(r^{2}+1)^{1\over 2}\over c})\over 4\pi\epsilon_{0}c^{2}(r^{2}+1)^{1\over 2}}\sum_{w=0}^{\infty}\sum_{s\leq m,s, odd}(-1)^{s-1\over 2}C^{m}_{s}{(-1)^{w}m^{2w+1}\over (2w+1)!(r^{2}+1)^{w+{1\over 2}}}$\\

$\sum_{v=0}^{2w+1}C^{2w+1}_{v}x^{2w+1-v}y^{v}\int_{-\pi}^{\pi}(-sin(\theta),cos(\theta),0)cos^{m-s}(\theta)sin^{s}(\theta)cos^{2w+1-v}(\theta)sin^{v}(\theta)d\theta+O({1\over r^{2}})$\\

It follows that;\\

for $m$ even;\\

$\overline{E}_{3}(\overline{r},t)=+{mcos(mt)sin(m{(r^{2}+1)^{1\over 2}\over c})\over 4\pi\epsilon_{0}c^{2}(r^{2}+1)^{1\over 2}}\sum_{w=0}^{\infty}\sum_{s\leq m,s, odd}(-1)^{s-1\over 2}C^{m}_{s}{(-1)^{w}m^{2w+1}\over (2w+1)!(r^{2}+1)^{w+{1\over 2}}}$\\

$(-\sum_{v=0,v,even}^{2w+1}C^{2w+1}_{v}x^{2w+1-v}y^{v}I_{2w+1+m-v-s}^{v+s+1},\sum_{v=0,v,odd}^{2w+1}C^{2w+1}_{v}x^{2w+1-v}y^{v}I_{2w+2+m-v-s}^{v+s},0)$\\

$-{msin(mt)cos(m{(r^{2}+1)^{1\over 2}\over c})\over 4\pi\epsilon_{0}c^{2}(r^{2}+1)^{1\over 2}}\sum_{w=0}^{\infty}\sum_{s\leq m,s, odd}(-1)^{s-1\over 2}C^{m}_{s}{(-1)^{w}m^{2w+1}\over (2w+1)!(r^{2}+1)^{w+{1\over 2}}}$\\

$(-\sum_{v=0,v,even}^{2w+1}C^{2w+1}_{v}x^{2w+1-v}y^{v}I_{2w+1+m-v-s}^{v+s+1},\sum_{v=0,v,odd}^{2w+1}C^{2w+1}_{v}x^{2w+1-v}y^{v}I_{2w+2+m-v-s}^{v+s},0)+O({1\over r^{2}})$ $(*****)$\\

for $m$ odd;\\

$\overline{E}_{3}(\overline{r},t)={-mcos(mt)cos(m{(r^{2}+1)^{1\over 2}\over c})\over 4\pi\epsilon_{0}c^{2}(r^{2}+1)^{1\over 2}}\sum_{w=0}^{\infty}\sum_{s\leq m,s, odd}(-1)^{s-1\over 2}C^{m}_{s}{(-1)^{w}m^{2w}\over (2w)!(r^{2}+1)^{w}}$\\

$(-\sum_{v=0,v,even}^{2w}C^{2w}_{v}x^{2w-v}y^{v}I_{2w+m-v-s}^{v+s+1},\sum_{v=0,v,odd}^{2w}C^{2w}_{v}x^{2w-v}y^{v}I_{2w+1+m-v-s}^{v+s},0)$\\

$-{msin(mt)sin(m{(r^{2}+1)^{1\over 2}\over c})\over 4\pi\epsilon_{0}c^{2}(r^{2}+1)^{1\over 2}}\sum_{w=0}^{\infty}\sum_{s\leq m,s, odd}(-1)^{s-1\over 2}C^{m}_{s}{(-1)^{w}m^{2w}\over (2w)!(r^{2}+1)^{w}}$\\

$(-\sum_{v=0,v,even}^{2w}C^{2w}_{v}x^{2w-v}y^{v}I_{2w+m-v-s}^{v+s+1},\sum_{v=0,v,odd}^{2w}C^{2w}_{v}x^{2w-v}y^{v}I_{2w+1+m-v-s}^{v+s},0)+O({1\over r^{2}})$ $(******)$\\

We have that;\\

$\overline{B}_{2}(\overline{r},t)={\mu_{0}\over 4\pi}\int [{\dot{\overline{J}}(\overline{r'},t_{r})\over c\mathfrak{r}}]\times \hat{\mathfrak{\overline{r}}} d\tau'$, $(*)$\\

$(-sin(\theta),cos(\theta),0)\times (x-cos(\theta),y-sin(\theta),z)$\\

$=(cos(\theta)z,sin(\theta)z,-sin(\theta)(y-sin(\theta)-cos(\theta)(x-cos(\theta)))$\\

$\overline{B}_{2}(\overline{r},t)={\mu_{0}\over 4\pi c}\int_{-\pi}^{\pi}{msin(m\theta)cos(mt_{r})\over (r^{2}+1-2xcos(\theta)-2ysin(\theta))}(cos(\theta)z,sin(\theta)z,-sin(\theta)(y$\\

$-sin(\theta)-cos(\theta)(x-cos(\theta))))d\theta$ $(*)$\\

$={\mu_{0}\over 4\pi c(r^{2}+1)}\int_{-\pi}^{\pi}{msin(m\theta)cos(m(t-{\mathfrak{r}\over c}))\over (1-{2xcos(\theta)-2ysin(\theta)\over (r^{2}+1)})}(cos(\theta)z,sin(\theta)z,-sin(\theta)y-cos(\theta)x)d\theta+O({1\over r^{2}})$\\

$={\mu_{0}\over 4\pi c(r^{2}+1)}\int_{-\pi}^{\pi} msin(m\theta)cos(m(t-{\mathfrak{r}\over c}))(cos(\theta)z,sin(\theta)z,-sin(\theta)y-cos(\theta)x)d\theta+O({1\over r^{2}})$\\

$={\mu_{0}mcos(mt)\over 4\pi c(r^{2}+1)}\int_{-\pi}^{\pi} sin(m\theta)cos(m{\mathfrak{r}\over c}))(cos(\theta)z,sin(\theta)z,-sin(\theta)y-cos(\theta)x)d\theta$\\

$+{\mu_{0}msin(mt)\over 4\pi c(r^{2}+1)}\int_{-\pi}^{\pi} sin(m\theta)sin(m{\mathfrak{r}\over c}))(cos(\theta)z,sin(\theta)z,-sin(\theta)y-cos(\theta)x)d\theta+O({1\over r^{2}})$\\

$={\mu_{0}mcos(mt)cos(m(r^{2}+1)^{1\over 2})\over 4\pi c(r^{2}+1)}\int_{-\pi}^{\pi}sin(m\theta)cos({mxcos(\theta)+mysin(\theta)\over (r^{2}+1)^{1\over 2}})(cos(\theta)z,sin(\theta)z,-sin(\theta)y-cos(\theta)x)d\theta$\\

$-{\mu_{0}mcos(mt)sin(m(r^{2}+1)^{1\over 2})\over 4\pi c(r^{2}+1)}\int_{-\pi}^{\pi}sin(m\theta)sin({mxcos(\theta)+mysin(\theta)\over (r^{2}+1)^{1\over 2}})(cos(\theta)z,sin(\theta)z,-sin(\theta)y-cos(\theta)x)d\theta$\\

$+{\mu_{0}msin(mt)sin(m(r^{2}+1)^{1\over 2})\over 4\pi c(r^{2}+1)}\int_{-\pi}^{\pi}sin(m\theta)cos({mxcos(\theta)+mysin(\theta)\over (r^{2}+1)^{1\over 2}})(cos(\theta)z,sin(\theta)z,-sin(\theta)y-cos(\theta)x)d\theta$\\

$+{\mu_{0}msin(mt)cos(m(r^{2}+1)^{1\over 2})\over 4\pi c(r^{2}+1)}\int_{-\pi}^{\pi}sin(m\theta)sin({mxcos(\theta)+mysin(\theta)\over (r^{2}+1)^{1\over 2}})(cos(\theta)z,sin(\theta)z,-sin(\theta)y-cos(\theta)x)d\theta+O({1\over r^{2}})$ $(*******)$\\

$={\mu_{0}mcos(mt)cos(m{(r^{2}+1)^{1\over 2}\over c})\over 4\pi c(r^{2}+1)}\sum_{w=0}^{\infty}\sum_{s\leq m,s, odd}(-1)^{s-1\over 2}C^{m}_{s}{(-1)^{w}m^{2w}\over (2w)!(r^{2}+1)^{w}}$\\

$\sum_{v=0}^{2w}C^{2w}_{v}x^{2w-v}y^{v}\int_{-\pi}^{\pi}(zcos(\theta),zsin(\theta),-ysin(\theta)-xcos(\theta))$\\

$cos^{m-s}(\theta)sin^{s}(\theta)cos^{2w-v}(\theta)sin^{v}(\theta)d\theta$\\

$-{\mu_{0}mcos(mt)sin(m{(r^{2}+1)^{1\over 2}\over c})\over 4\pi c(r^{2}+1)}\sum_{w=0}^{\infty}\sum_{s\leq m,s, odd}(-1)^{s-1\over 2}C^{m}_{s}{(-1)^{w}m^{2w+1}\over (2w+1)!(r^{2}+1)^{w+{1\over 2}}}$\\

$\sum_{v=0}^{2w+1}C^{2w+1}_{v}x^{2w+1-v}y^{v}\int_{-\pi}^{\pi}(zcos(\theta),zsin(\theta),-ysin(\theta)-xcos(\theta))$\\

$cos^{m-s}(\theta)sin^{s}(\theta)cos^{2w+1-v}(\theta)sin^{v}(\theta)d\theta$\\

$+{\mu_{0}msin(mt)sin(m{(r^{2}+1)^{1\over 2}\over c})\over 4\pi c(r^{2}+1)}\sum_{w=0}^{\infty}\sum_{s\leq m,s, odd}(-1)^{s-1\over 2}C^{m}_{s}{(-1)^{w}m^{2w}\over (2w)!(r^{2}+1)^{w}}$\\

$\sum_{v=0}^{2w}C^{2w}_{v}x^{2w-v}y^{v}\int_{-\pi}^{\pi}(zcos(\theta),zsin(\theta),-ysin(\theta)-xcos(\theta))$\\

$cos^{m-s}(\theta)sin^{s}(\theta)cos^{2w-v}(\theta)sin^{v}(\theta)d\theta$\\

$+{\mu_{0}msin(mt)cos(m{(r^{2}+1)^{1\over 2}\over c})\over 4\pi c(r^{2}+1)}\sum_{w=0}^{\infty}\sum_{s\leq m,s, odd}(-1)^{s-1\over 2}C^{m}_{s}{(-1)^{w}m^{2w+1}\over (2w+1)!(r^{2}+1)^{w+{1\over 2}}}$\\

$\sum_{v=0}^{2w+1}C^{2w+1}_{v}x^{2w+1-v}y^{v}\int_{-\pi}^{\pi}(zcos(\theta),zsin(\theta),-ysin(\theta)-xcos(\theta))$\\

$cos^{m-s}(\theta)sin^{s}(\theta)cos^{2w+1-v}(\theta)sin^{v}(\theta)d\theta+O({1\over r^{2}})$\\

It follows that;\\

for $m$ even;\\

$\overline{B}_{2}(\overline{r},t)=-{\mu_{0}mcos(mt)sin(m{(r^{2}+1)^{1\over 2}\over c})\over 4\pi c(r^{2}+1)}\sum_{w=0}^{\infty}\sum_{s\leq m,s, odd}(-1)^{s-1\over 2}C^{m}_{s}{(-1)^{w}m^{2w+1}\over (2w+1)!(r^{2}+1)^{w+{1\over 2}}}$\\

$(\sum_{v=0,v,odd}^{2w+1}C^{2w+1}_{v}x^{2w+1-v}y^{v}zI_{2w+2+m-v-s}^{s+v},\sum_{v=0,v,even}^{2w+1}C^{2w+1}_{v}x^{2w+1-v}y^{v}zI_{2w+1+m-v-s}^{s+v+1},$\\

$-\sum_{v=0,v,even}^{2w+1}C^{2w+1}_{v}x^{2w+1-v}y^{v+1}I_{2w+1+m-v-s}^{s+v+1}-\sum_{v=0,v,odd}^{2w+1}C^{2w+1}_{v}x^{2w+2-v}y^{v}I_{2w+2+m-v-s}^{s+v})$\\

$+{\mu_{0}sin(mt)cos(m{(r^{2}+1)^{1\over 2}\over c})\over 4\pi c(r^{2}+1)}\sum_{w=0}^{\infty}\sum_{s\leq m,s, odd}(-1)^{s-1\over 2}C^{m}_{s}{(-1)^{w}m^{2w+1}\over (2w+1)!(r^{2}+1)^{w+{1\over 2}}}$\\

$(\sum_{v=0,v,odd}^{2w+1}C^{2w+1}_{v}x^{2w+1-v}y^{v}zI_{2w+2+m-v-s}^{s+v},\sum_{v=0,v,even}^{2w+1}C^{2w+1}_{v}x^{2w+1-v}y^{v}zI_{2w+1+m-v-s}^{s+v+1},$\\

$-\sum_{v=0,v,even}^{2w+1}C^{2w+1}_{v}x^{2w+1-v}y^{v+1}I_{2w+1+m-v-s}^{s+v+1}-\sum_{v=0,v,odd}^{2w+1}C^{2w+1}_{v}x^{2w+2-v}y^{v}I_{2w+2+m-v-s}^{s+v})+O({1\over r^{2}})$ $(********)$\\

and for $m$ odd;\\

$\overline{B}_{2}(\overline{r},t)={\mu_{0}mcos(mt)cos(m{(r^{2}+1)^{1\over 2}\over c})\over 4\pi c(r^{2}+1)}\sum_{w=0}^{\infty}\sum_{s\leq m,s, odd}(-1)^{s-1\over 2}C^{m}_{s}{(-1)^{w}m^{2w}\over (2w)!(r^{2}+1)^{w}}$\\

$(\sum_{v=0,v,odd}^{2w}C^{2w}_{v}x^{2w-v}y^{v}zI_{2w+1+m-v-s}^{s+v},\sum_{v=0,v,even}^{2w}C^{2w}_{v}x^{2w-v}y^{v}zI_{2w+m-v-s}^{s+v+1},$\\

$-\sum_{v=0,v,even}^{2w}C^{2w}_{v}x^{2w-v}y^{v+1}I_{2w+m-v-s}^{s+v+1}-\sum_{v=0,v,odd}^{2w}C^{2w}_{v}x^{2w+1-v}y^{v}I_{2w+1+m-v-s}^{s+v})$\\

$+{\mu_{0}sin(mt)sin(m{(r^{2}+1)^{1\over 2}\over c})\over 4\pi c(r^{2}+1)}\sum_{w=0}^{\infty}\sum_{s\leq m,s, odd}(-1)^{s-1\over 2}C^{m}_{s}{(-1)^{w}m^{2w+1}\over (2w)!(r^{2}+1)^{w}}$\\

$(\sum_{v=0,v,odd}^{2w}C^{2w}_{v}x^{2w-v}y^{v}zI_{2w+1+m-v-s}^{s+v},\sum_{v=0,v,even}^{2w}C^{2w}_{v}x^{2w-v}y^{v}zI_{2w+m-v-s}^{s+v+1},$\\

$-\sum_{v=0,v,even}^{2w}C^{2w}_{v}x^{2w-v}y^{v+1}I_{2w+m-v-s}^{s+v+1}-\sum_{v=0,v,odd}^{2w}C^{2w}_{v}x^{2w+1-v}y^{v}I_{2w+1+m-v-s}^{s+v})+O({1\over r^{2}})$ $(*********)$\\

We now compute the Poynting vectors $\overline{E}_{2}\times \overline{B}_{2}$ and $\overline{E}_{3}\times \overline{B}_{2}$ in the cases when $m$ is even and $m$ is odd. If $m$ is even, by $(*)$ and $(**)$ we have that;\\

$\overline{E}_{2,e}=\overline{E}_{2,e}^{1}+\overline{E}_{2,e}^{2}$\\

$=-\alpha msin(mt)cos(m{(r^{2}+1)^{1\over 2}\over c})\Gamma+\alpha mcos(mt)sin(m{(r^{2}+1)^{1\over 2}\over c})\Gamma$\\

where $\alpha={1\over 4\pi\epsilon_{0}c(r^{2}+1)}$ and;\\

$\Gamma=\int_{-\pi}^{\pi}cos(m\theta)cos({mxcos(\theta)+mysin(\theta)\over (r^{2}+1)^{1\over 2}})(x,y,z)d\theta$\\

If $m$ is even, by $(*******),(********)$;\\

$\overline{B}_{2,e}=\overline{B}_{2,e}^{1}+\overline{B}_{2,e}^{2}$\\

$=-\beta mcos(mt)sin(m{(r^{2}+1)^{1\over 2}\over c})\Gamma'+\beta msin(mt)cos(m{(r^{2}+1)^{1\over 2}\over c})\Gamma'+O({1\over r^{2}})$\\

where $\beta={\mu_{0}\over 4\pi c(r^{2}+1)}$ and;\\

$\Gamma'=\int_{-\pi}^{\pi}sin(m\theta)sin({mxcos(\theta)+mysin(\theta)\over (r^{2}+1)^{1\over 2}})(cos(\theta)z,sin(\theta)z,-sin(\theta)y-cos(\theta)x)d\theta$\\

It follows that;\\

$\overline{E}_{2,e}\times \overline{B}_{2,e}$\\

$=\overline{E}_{2,e}^{1}\times \overline{B}_{2,e}^{1}+\overline{E}_{2,e}^{2}\times \overline{B}_{2,e}^{1}+\overline{E}_{2,e}^{1}\times\overline{B}_{2,e}^{2}+\overline{E}_{2,e}^{2}\times\overline{B}_{2,e}^{2}+O({1\over r^{3}})$\\

$=\alpha\beta m^{2}sin(mt)cos(mt)sin(m{(r^{2}+1)^{1\over 2}\over c})cos(m{(r^{2}+1)^{1\over 2}\over c})\Gamma\times\Gamma'$\\

$-\alpha\beta m^{2}cos^{2}(mt)sin^{2}(m{(r^{2}+1)^{1\over 2}\over c})\Gamma\times\Gamma'$\\

$-\alpha\beta m^{2}sin^{2}(mt)cos^{2}(m{(r^{2}+1)^{1\over 2}\over c})\Gamma\times\Gamma'$\\

$+\alpha\beta m^{2}sin(mt)cos(mt)sin(m{(r^{2}+1)^{1\over 2}\over c})cos(m{(r^{2}+1)^{1\over 2}\over c})\Gamma\times\Gamma'+O({1\over r^{3}})$\\

$=\alpha\beta m^{2}[-sin^{2}(mt)cos^{2}(m{(r^{2}+1)^{1\over 2}\over c})-cos^{2}(mt)sin^{2}(m{(r^{2}+1)^{1\over 2}\over c})$\\

$+2sin(mt)cos(mt)sin(m{(r^{2}+1)^{1\over 2}\over c})cos(m{(r^{2}+1)^{1\over 2}\over c})]\Gamma\times\Gamma'+O({1\over r^{3}})$\\

Similarly, by $(****),(*****)$;\\

$\overline{E}_{3,e}=\overline{E}_{3,e}^{1}+\overline{E}_{3,e}^{2}$\\

$=\gamma mcos(mt)sin(m{(r^{2}+1)^{1\over 2}\over c})\Gamma''-\gamma msin(mt)cos(m{(r^{2}+1)^{1\over 2}\over c})\Gamma''+O({1\over r^{2}})$\\

where $\gamma={1\over 4\pi\epsilon_{0}c^{2}(r^{2}+1)^{1\over 2}}$ and;\\

$\Gamma''=\int_{-\pi}^{\pi}sin(m\theta)sin({mxcos(\theta)+mysin(\theta)\over (r^{2}+1)^{1\over 2}})(-sin(\theta),cos(\theta),0)d\theta$\\

If $m$ is even, it follows that;\\

$\overline{E}_{3,e}\times \overline{B}_{2,e}$\\

$=\overline{E}_{3,e}^{1}\times \overline{B}_{2,e}^{1}+\overline{E}_{3,e}^{2}\times \overline{B}_{2,e}^{1}+\overline{E}_{3,e}^{1}\times\overline{B}_{2,e}^{2}+\overline{E}_{3,e}^{2}\times\overline{B}_{2,e}^{2}+O({1\over r^{3}})$\\

$=-\beta\gamma m^{2}cos^{2}(mt)sin^{2}(m{(r^{2}+1)^{1\over 2}\over c})\Gamma''\times\Gamma'$\\

$+\beta\gamma m^{2}sin(mt)cos(mt)sin(m{(r^{2}+1)^{1\over 2}\over c})cos(m{(r^{2}+1)^{1\over 2}\over c})\Gamma''\times\Gamma'$\\

$+\beta\gamma m^{2}sin(mt)cos(mt)sin(m{(r^{2}+1)^{1\over 2}\over c})cos(m{(r^{2}+1)^{1\over 2}\over c})\Gamma''\times\Gamma'$\\

$-\beta\gamma m^{2}sin^{2}(mt)cos^{2}(m{(r^{2}+1)^{1\over 2}\over c})\Gamma''\times\Gamma'+O({1\over r^{3}})$\\

$=\beta\gamma m^{2}[-sin^{2}(mt)cos^{2}(m{(r^{2}+1)^{1\over 2}\over c})-cos^{2}(mt)sin^{2}(m{(r^{2}+1)^{1\over 2}\over c})$\\

$+2sin(mt)cos(mt)sin(m{(r^{2}+1)^{1\over 2}\over c})cos(m{(r^{2}+1)^{1\over 2}\over c})]\Gamma''\times\Gamma'+O({1\over r^{3}})$\\

If $m$ is odd, by $(*)$ and $(***)$ we have that;\\

$\overline{E}_{2,o}=\overline{E}_{2,o}^{1}+\overline{E}_{2,o}^{2}+O({1\over r^{2}})$\\

$=\alpha msin(mt)sin(m{(r^{2}+1)^{1\over 2}\over c})\Gamma'''+\alpha mcos(mt)cos(m{(r^{2}+1)^{1\over 2}\over c})\Gamma'''$\\

where $\Gamma'''=\int_{-\pi}^{\pi}cos(m\theta)sin({mxcos(\theta)+mysin(\theta)\over (r^{2}+1)^{1\over 2}})(x,y,z)d\theta$\\

If $m$ is odd, by $(*******),(*********)$;\\

$\overline{B}_{2,o}=\overline{B}_{2,o}^{1}+\overline{B}_{2,o}^{2}+O({1\over r^{2}})$\\

$=\beta mcos(mt)cos(m{(r^{2}+1)^{1\over 2}\over c})\Gamma''''+\beta msin(mt)sin(m{(r^{2}+1)^{1\over 2}\over c})\Gamma''''+O({1\over r^{2}})$\\

where;\\

$\Gamma''''=\int_{-\pi}^{\pi}sin(m\theta)cos({mxcos(\theta)+mysin(\theta)\over (r^{2}+1)^{1\over 2}})(cos(\theta)z,sin(\theta)z,-sin(\theta)y-cos(\theta)x)d\theta$\\

It follows that;\\

$\overline{E}_{2,o}\times \overline{B}_{2,o}$\\

$=\overline{E}_{2,o}^{1}\times \overline{B}_{2,o}^{1}+\overline{E}_{2,o}^{2}\times \overline{B}_{2,o}^{1}+\overline{E}_{2,o}^{1}\times\overline{B}_{2,o}^{2}+\overline{E}_{2,o}^{2}\times\overline{B}_{2,o}^{2}+O({1\over r^{3}})$\\

$=\alpha\beta m^{2}sin(mt)cos(mt)sin(m{(r^{2}+1)^{1\over 2}\over c})cos(m{(r^{2}+1)^{1\over 2}\over c})\Gamma'''\times\Gamma''''$\\

$+\alpha\beta m^{2}cos^{2}(mt)cos^{2}(m{(r^{2}+1)^{1\over 2}\over c})\Gamma'''\times\Gamma''''$\\

$+\alpha\beta m^{2}sin^{2}(mt)sin^{2}(m{(r^{2}+1)^{1\over 2}\over c})\Gamma'''\times\Gamma''''$\\

$+\alpha\beta m^{2}sin(mt)cos(mt)sin(m{(r^{2}+1)^{1\over 2}\over c})cos(m{(r^{2}+1)^{1\over 2}\over c})\Gamma'''\times\Gamma''''+O({1\over r^{3}})$\\

$=\alpha\beta m^{2}[cos^{2}(mt)cos^{2}(m{(r^{2}+1)^{1\over 2}\over c})+sin^{2}(mt)sin^{2}(m{(r^{2}+1)^{1\over 2}\over c})$\\

$+2sin(mt)cos(mt)sin(m{(r^{2}+1)^{1\over 2}\over c})cos(m{(r^{2}+1)^{1\over 2}\over c})]\Gamma'''\times\Gamma''''+O({1\over r^{3}})$\\

Similarly, by $(****),(******)$;\\

$\overline{E}_{3,o}=\overline{E}_{3,o}^{1}+\overline{E}_{3,o}^{2}$\\

$=-\gamma mcos(mt)cos(m{(r^{2}+1)^{1\over 2}\over c})\Gamma'''''-\gamma msin(mt)sin(m{(r^{2}+1)^{1\over 2}\over c})\Gamma'''''+O({1\over r^{2}})$\\

where;\\

$\Gamma'''''=\int_{-\pi}^{\pi}sin(m\theta)cos({mxcos(\theta)+mysin(\theta)\over (r^{2}+1)^{1\over 2}})(-sin(\theta),cos(\theta),0)d\theta$\\

If $m$ is odd, it follows that;\\

$\overline{E}_{3,o}\times \overline{B}_{2,o}$\\

$=\overline{E}_{3,o}^{1}\times \overline{B}_{2,o}^{1}+\overline{E}_{3,o}^{2}\times \overline{B}_{2,o}^{1}+\overline{E}_{3,o}^{1}\times\overline{B}_{2,o}^{2}+\overline{E}_{3,o}^{2}\times\overline{B}_{2,o}^{2}+O({1\over r^{3}})$\\

$=-\beta\gamma m^{2}cos^{2}(mt)cos^{2}(m{(r^{2}+1)^{1\over 2}\over c})\Gamma'''''\times\Gamma''''$\\

$-\beta\gamma m^{2}sin(mt)cos(mt)sin(m{(r^{2}+1)^{1\over 2}\over c})cos(m{(r^{2}+1)^{1\over 2}\over c})\Gamma'''''\times\Gamma''''$\\

$-\beta\gamma m^{2}sin(mt)cos(mt)sin(m{(r^{2}+1)^{1\over 2}\over c})cos(m{(r^{2}+1)^{1\over 2}\over c})\Gamma'''''\times\Gamma''''$\\

$-\beta\gamma m^{2}sin^{2}(mt)sin^{2}(m{(r^{2}+1)^{1\over 2}\over c})\Gamma'''''\times\Gamma''''+O({1\over r^{3}})$\\

$=\beta\gamma m^{2}[-cos^{2}(mt)cos^{2}(m{(r^{2}+1)^{1\over 2}\over c})-sin^{2}(mt)sin^{2}(m{(r^{2}+1)^{1\over 2}\over c})$\\

$-2sin(mt)cos(mt)sin(m{(r^{2}+1)^{1\over 2}\over c})cos(m{(r^{2}+1)^{1\over 2}\over c})]\Gamma'''''\times\Gamma''''+O({1\over r^{3}})$\\

We compute $\Gamma\times\Gamma'$. By $(**)$ and $(********)$, we have that;\\

$\Gamma=\sum_{w=0}^{\infty}\sum_{s\leq m,s, even}(-1)^{s\over 2}C^{m}_{s}{(-1)^{w}m^{2w}\over (2w)!(r^{2}+1)^{w}}$\\

$\sum_{v=0,v,even}^{2w}C^{2w}_{v}x^{2w-v}y^{v}(x,y,z)I_{2w+m-s-v,s+v}$\\

$\Gamma'=\sum_{w=0}^{\infty}\sum_{s\leq m,s, odd}(-1)^{s-1\over 2}C^{m}_{s}{(-1)^{w}m^{2w+1}\over (2w+1)!(r^{2}+1)^{w+{1\over 2}}}$\\

$(\sum_{v=0,v,odd}^{2w+1}C^{2w+1}_{v}x^{2w+1-v}y^{v}zI_{2w+2+m-v-s}^{s+v},\sum_{v=0,v,even}^{2w+1}C^{2w+1}_{v}x^{2w+1-v}y^{v}zI_{2w+1+m-v-s}^{s+v+1},$\\

$-\sum_{v=0,v,even}^{2w+1}C^{2w+1}_{v}x^{2w+1-v}y^{v+1}I_{2w+1+m-v-s}^{s+v+1}-\sum_{v=0,v,odd}^{2w+1}C^{2w+1}_{v}x^{2w+2-v}y^{v}I_{2w+2+m-v-s}^{s+v})$\\

so that;\\

$\Gamma\times\Gamma'=\sum_{w=0}^{\infty}\sum_{w'=0}^{\infty}\sum_{s\leq m,s, even,s'\leq m,s',odd}(-1)^{s\over 2}C^{m}_{s}{(-1)^{w}m^{2w}\over (2w)!(r^{2}+1)^{w}}(-1)^{s'-1\over 2}C^{m}_{s'}{(-1)^{w'}m^{2w'+1}\over (2w'+1)!(r^{2}+1)^{w'+{1\over 2}}}$\\

$(-\sum_{v=0,v,even,v'=0,v',even}^{2w,2w'+1}C^{2w}_{v}C^{2w'+1}_{v'}x^{2w+2w'+1-v-v'}y^{v+v'+2}I_{2w+m-s-v}^{s+v}I_{2w'+1+m-v'-s'}^{s'+v'+1}$\\

$-\sum_{v=0,v,even,v'=0,v',odd}^{2w,2w'+1}C^{2w}_{v}C^{2w'+1}_{v'}x^{2w+2w'+2-v-v'}y^{v+v'+1}I_{2w+m-s-v}^{s+v}I_{2w'+2-m-v'-s'}^{s'+v'}$\\

$-\sum_{v=0,v,even,v'=0,v',even}^{2w,2w'+1}C^{2w}_{v}C^{2w'+1}_{v'}x^{2w+2w'+1-v-v'}y^{v+v'}z^{2}I_{2w+m-s-v}^{s+v}I_{2w'+1-m-v'-s'}^{s'+v'+1},$\\

$+\sum_{v=0,v,even,v'=0,v',even}^{2w,2w'+1}C^{2w}_{v}C^{2w'+1}_{v'}x^{2w+2w'+2-v-v'}y^{v+v'+1}I_{2w+m-s-v}^{s+v}I_{2w'+1+m-v'-s'}^{s'+v'+1}$\\

$+\sum_{v=0,v,even,v'=0,v',odd}^{2w,2w'+1}C^{2w}_{v}C^{2w'+1}_{v'}x^{2w+2w'+3-v-v'}y^{v+v'}I_{2w+m-s-v}^{s+v}I_{2w'+2+m-v'-s'}^{s'+v'}$\\

$+\sum_{v=0,v,even,v'=0,v',odd}^{2w,2w'+1}C^{2w}_{v}C^{2w'+1}_{v'}x^{2w+2w'+1-v-v'}y^{v+v'}z^{2}I_{2w+m-s-v}^{s+v}I_{2w'+2+m-v'-s'}^{s'+v'},$\\

$+\sum_{v=0,v,even,v'=0,v',even}^{2w,2w'+1}C^{2w}_{v}C^{2w'+1}_{v'}x^{2w+2w'+2-v-v'}y^{v+v'}zI_{2w+m-s-v}^{s+v}I_{2w'+1+m-v'-s'}^{s'+v'+1}$\\

$-\sum_{v=0,v,even,v'=0,v',odd}^{2w,2w'+1}C^{2w}_{v}C^{2w'+1}_{v'}x^{2w+2w'+1-v-v'}y^{v+v'+1}zI_{2w+m-s-v}^{s+v}I_{2w'+2+m-v'-s'}^{s'+v'})$ $(\dag)$\\

By $(*****)$, we have that;\\

$\Gamma''=\sum_{w=0}^{\infty}\sum_{s\leq m,s, odd}(-1)^{s-1\over 2}C^{m}_{s}{(-1)^{w}m^{2w+1}\over (2w+1)!(r^{2}+1)^{w+{1\over 2}}}$\\

$(-\sum_{v=0,v,even}^{2w+1}C^{2w+1}_{v}x^{2w+1-v}y^{v}I_{2w+1+m-v-s}^{v+s+1},\sum_{v=0,v,odd}^{2w+1}C^{2w+1}_{v}x^{2w+1-v}y^{v}I_{2w+2+m-v-s}^{v+s},0)$\\

It follows that;\\

$\Gamma''\times\Gamma'$\\

$=\sum_{w=0}^{\infty}\sum_{w'=0}^{\infty}\sum_{s\leq m,s,odd,s'\leq m,s',odd}(-1)^{s-1\over 2}C^{m}_{s}{(-1)^{w}m^{2w+1}\over (2w+1)!(r^{2}+1)^{w+{1\over 2}}}(-1)^{s'-1\over 2}C^{m}_{s'}{(-1)^{w'}m^{2w'+1}\over (2w'+1)!(r^{2}+1)^{w'+{1\over 2}}}$\\

$(-\sum_{v=0,v,odd,v'=0,v',even}^{2w+1,2w'+1}C^{2w+1}_{v}C^{2w'+1}_{v'}x^{2w+2w'+2-v-v'}y^{v+v'+1}I_{2w+2+m-s-v}^{s+v}I_{2w'+1+m-v'-s'}^{s'+v'+1}$\\

$-\sum_{v=0,v,odd,v'=0,v',odd}^{2w+1,2w'+1}C^{2w+1}_{v}C^{2w'+1}_{v'}x^{2w+2w'+3-v-v'}y^{v+v'}I_{2w+2+m-s-v}^{s+v}I_{2w'+2-m-v'-s'}^{s'+v'},$\\

$-\sum_{v=0,v,even,v'=0,v',even}^{2w+1,2w'+1}C^{2w+1}_{v}C^{2w'+1}_{v'}x^{2w+2w'+2-v-v'}y^{v+v'+1}I_{2w+1+m-s-v}^{s+v+1}I_{2w'+1+m-v'-s'}^{s'+v'+1}$\\

$-\sum_{v=0,v,even,v'=0,v',odd}^{2w+1,2w'+1}C^{2w+1}_{v}C^{2w'+1}_{v'}x^{2w+2w'+3-v-v'}y^{v+v'}I_{2w+1+m-s-v}^{s+v+1}I_{2w'+2+m-v'-s'}^{s'+v'},$\\

$-\sum_{v=0,v,even,v'=0,v',even}^{2w+1,2w'+1}C^{2w+1}_{v}C^{2w'+1}_{v'}x^{2w+2w'+2-v-v'}y^{v+v'}zI_{2w+1+m-s-v}^{s+v+1}I_{2w'+1+m-v'-s'}^{s'+v'+1}$\\

$-\sum_{v=0,v,odd,v'=0,v',odd}^{2w+1,2w'+1}C^{2w+1}_{v}C^{2w'+1}_{v'}x^{2w+2w'+2-v-v'}y^{v+v'}zI_{2w+1++m-s-v}^{s+v+1}I_{2w'+2+m-v'-s'}^{s'+v'})$ $(\dag\dag)$\\

We now compute the flux of the Poynting vectors over the sphere $S(r)$. We have that $x=rsin(\phi)cos(\theta)$, $y=rsin(\phi)sin(\theta)$, $z=rcos(\phi)$ with coordinates $0\leq \phi<\pi$ and $-\pi\leq \theta<\pi$. We have that;\\

$r_{\phi}\times r_{\theta}=(rcos(\phi)cos(\theta),rcos(\phi)sin(\theta),-rsin(\phi))\times (-rsin(\phi)sin(\theta),rsin(\phi)cos(\theta),0)$\\

$=r^{2}(sin^{2}(\phi)cos(\theta),sin^{2}(\phi)sin(\theta),sin(\phi)cos(\phi)$\\

so that;\\

$d\overline{S}=\hat{\overline{n}}dS=r^{2}(sin^{2}(\phi)cos(\theta),sin^{2}(\phi)sin(\theta),sin(\phi)cos(\phi))d\phi d\theta$\\

and;\\

$\int_{S(r)}(\Gamma\times\Gamma')\centerdot d\overline{S}$\\

$=\int_{0}^{\pi}\int_{-\pi}^{\pi}(\Gamma\times\Gamma')|_{(rsin(\phi)cos(\theta),rsin(\phi)sin(\theta),rcos(\phi))}\centerdot r^{2}(sin^{2}(\phi)cos(\theta),$\\

$sin^{2}(\phi)sin(\theta),sin(\phi)cos(\phi)d\theta d\phi$\\

Applying this to $(\dag)$ gives;\\

$\int_{S(r)}(\Gamma\times\Gamma')\centerdot d\overline{S}$\\

$=\sum_{w=0}^{\infty}\sum_{w'=0}^{\infty}\sum_{s\leq m,s, even,s'\leq m,s',odd}(-1)^{s\over 2}C^{m}_{s}{(-1)^{w}m^{2w}\over (2w)!(r^{2}+1)^{w}}(-1)^{s'-1\over 2}C^{m}_{s'}{(-1)^{w'}m^{2w'+1}\over (2w'+1)!(r^{2}+1)^{w'+{1\over 2}}}$\\

$(-\sum_{v=0,v,even,v'=0,v',even}^{2w,2w'+1}C^{2w}_{v}C^{2w'+1}_{v'}I_{2w+m-s-v}^{s+v}I_{2w'+1+m-v'-s'}^{s'+v'+1}$\\

$\int_{0}^{\pi}\int_{-\pi}^{\pi}r^{2w+2w'+5}sin^{2w+2w'+5}(\phi)cos^{2w+2w'+2-v-v'}(\theta)sin^{v+v'+2}(\theta)d\theta d\phi$\\

$-\sum_{v=0,v,even,v'=0,v',odd}^{2w,2w'+1}C^{2w}_{v}C^{2w'+1}_{v'}I_{2w+m-s-v}^{s+v}I_{2w'+2-m-v'-s'}^{s'+v'}$\\

$\int_{0}^{\pi}\int_{-\pi}^{\pi}r^{2w+2w'+5}sin^{2w+2w'+5}(\phi)cos^{2w+2w'+3-v-v'}(\theta)sin^{v+v'+1}(\theta)d\theta d\phi$\\

$-\sum_{v=0,v,even,v'=0,v',even}^{2w,2w'+1}C^{2w}_{v}C^{2w'+1}_{v'}I_{2w+m-s-v}^{s+v}I_{2w'+1-m-v'-s'}^{s'+v'+1}$\\

$\int_{0}^{\pi}\int_{-\pi}^{\pi}r^{2w+2w'+5}sin^{2w+2w'+3}cos^{2}(\phi)cos^{2w+2w'+2-v-v'}(\theta)sin^{v+v'}(\theta)d\theta d\phi$\\

$+\sum_{v=0,v,even,v'=0,v',even}^{2w,2w'+1}C^{2w}_{v}C^{2w'+1}_{v'}I_{2w+m-s-v}^{s+v}I_{2w'+1+m-v'-s'}^{s'+v'+1}$\\

$\int_{0}^{\pi}\int_{-\pi}^{\pi}r^{2w+2w'+5}sin^{2w+2w'+5}(\phi)cos^{2w+2w'+2-v-v'}(\theta)sin^{v+v'+2}(\theta)d\theta d\phi$\\

$+\sum_{v=0,v,even,v'=0,v',odd}^{2w,2w'+1}C^{2w}_{v}C^{2w'+1}_{v'}I_{2w+m-s-v}^{s+v}I_{2w'+2+m-v'-s'}^{s'+v'}$\\

$\int_{0}^{\pi}\int_{-\pi}^{\pi}r^{2w+2w'+5}sin^{2w+2w'+5}(\phi)cos^{2w+2w'+3-v-v'}(\theta)sin^{v+v'+1}(\theta)d\theta d\phi$\\

$+\sum_{v=0,v,even,v'=0,v',odd}^{2w,2w'+1}C^{2w}_{v}C^{2w'+1}_{v'}I_{2w+m-s-v}^{s+v}I_{2w'+2+m-v'-s'}^{s'+v'}$\\

$\int_{0}^{\pi}\int_{-\pi}^{\pi}r^{2w+2w'+5}sin^{2w+2w'+3}(\phi)cos^{2}(\phi)cos^{2w+2w'+1-v-v'}(\theta)sin^{v+v'+1}(\theta)d\theta d\phi$\\

$+\sum_{v=0,v,even,v'=0,v',even}^{2w,2w'+1}C^{2w}_{v}C^{2w'+1}_{v'}I_{2w+m-s-v}^{s+v}I_{2w'+1+m-v'-s'}^{s'+v'+1}$\\

$\int_{0}^{\pi}\int_{-\pi}^{\pi}r^{2w+2w'+5}sin^{2w+2w'+3}(\phi)cos^{2}(\phi)cos^{2w+2w'+2-v-v'}(\theta)sin^{v+v'}(\theta)d\theta d\phi$\\

$-\sum_{v=0,v,even,v'=0,v',odd}^{2w,2w'+1}C^{2w}_{v}C^{2w'+1}_{v'}I_{2w+m-s-v}^{s+v}I_{2w'+2+m-v'-s'}^{s'+v'}$\\

$\int_{0}^{\pi}\int_{-\pi}^{\pi}r^{2w+2w'+5}sin^{2w+2w'+3}(\phi)cos^{2}(\phi)cos^{2w+2w'+1-v-v'}(\theta)sin^{v+v'+1}(\theta)d\theta d\phi$\\

$(\dag\dag\dag)$\\

An inspection of $(\dag\dag\dag)$ shows that the terms 1 and 4, 2 and 5, 3 and 7, and 6 and 8 cancel. This proves that;\\

$\int_{S(r)}(\Gamma\times\Gamma')\centerdot d\overline{S}=0$\\

Applying the same method to $(\dag\dag)$, we have that;\\

$\int_{S(r)}(\Gamma''\times\Gamma')\centerdot d\overline{S}$\\

$=\sum_{w=0}^{\infty}\sum_{w'=0}^{\infty}\sum_{s\leq m,s,odd,s'\leq m,s',odd}(-1)^{s-1\over 2}C^{m}_{s}{(-1)^{w}m^{2w+1}\over (2w+1)!(r^{2}+1)^{w+{1\over 2}}}(-1)^{s'-1\over 2}C^{m}_{s'}{(-1)^{w'}m^{2w'+1}\over (2w'+1)!(r^{2}+1)^{w'+{1\over 2}}}$\\

$-\sum_{v=0,v,odd,v'=0,v',even}^{2w+1,2w'+1}C^{2w+1}_{v}C^{2w'+1}_{v'}I_{2w+2+m-s-v}^{s+v}I_{2w'+1+m-v'-s'}^{s'+v'+1}$\\

$\int_{0}^{\pi}\int_{-\pi}^{\pi}r^{2w+2w'+5}sin^{2w+2w'+5}(\phi)cos^{2w+2w'+3-v-v'}(\theta)sin^{v+v'+1}(\theta)d\theta d\phi$\\

$-\sum_{v=0,v,odd,v'=0,v',odd}^{2w+1,2w'+1}C^{2w+1}_{v}C^{2w'+1}_{v'}I_{2w+2+m-s-v}^{s+v}I_{2w'+2-m-v'-s'}^{s'+v'}$\\

$\int_{0}^{\pi}\int_{-\pi}^{\pi}r^{2w+2w'+5}sin^{2w+2w'+5}(\phi)cos^{2w+2w'+4-v-v'}(\theta)sin^{v+v'}(\theta)d\theta d\phi$\\

$-\sum_{v=0,v,even,v'=0,v',even}^{2w+1,2w'+1}C^{2w+1}_{v}C^{2w'+1}_{v'}I_{2w+1+m-s-v}^{s+v+1}I_{2w'+1+m-v'-s'}^{s'+v'+1}$\\

$\int_{0}^{\pi}\int_{-\pi}^{\pi}r^{2w+2w'+5}sin^{2w+2w'+5}(\phi)cos^{2w+2w'+2-v-v'}(\theta)sin^{v+v'+2}(\theta)d\theta d\phi$\\

$-\sum_{v=0,v,even,v'=0,v',odd}^{2w+1,2w'+1}C^{2w+1}_{v}C^{2w'+1}_{v'}I_{2w+1+m-s-v}^{s+v+1}I_{2w'+2+m-v'-s'}^{s'+v'}$\\

$\int_{0}^{\pi}\int_{-\pi}^{\pi}r^{2w+2w'+5}sin^{2w+2w'+5}(\phi)cos^{2w+2w'+3-v-v'}(\theta)sin^{v+v'+1}(\theta)d\theta d\phi$\\

$-\sum_{v=0,v,even,v'=0,v',even}^{2w+1,2w'+1}C^{2w+1}_{v}C^{2w'+1}_{v'}I_{2w+1+m-s-v}^{s+v+1}I_{2w'+1+m-v'-s'}^{s'+v'+1}$\\

$\int_{0}^{\pi}\int_{-\pi}^{\pi}r^{2w+2w'+5}sin^{2w+2w'+3}(\phi)cos^{2}(\phi)cos^{2w+2w'+2-v-v'}(\theta)sin^{v+v'}(\theta)d\theta d\phi$\\

$-\sum_{v=0,v,odd,v'=0,v',odd}^{2w+1,2w'+1}C^{2w+1}_{v}C^{2w'+1}_{v'}I_{2w+1++m-s-v}^{s+v+1}I_{2w'+2+m-v'-s'}^{s'+v'})$\\

$\int_{0}^{\pi}\int_{-\pi}^{\pi}r^{2w+2w'+5}sin^{2w+2w'+3}(\phi)cos^{2}(\phi)cos^{2w+2w'+2-v-v'}(\theta)sin^{v+v'}(\theta)d\theta d\phi$\\

$(\dag\dag\dag\dag)$\\

Using the notation $I_{\alpha}^{\beta}$ again, and letting;\\

$J_{\gamma}={2^{\gamma+1}({\gamma-1\over 2})!^{2}\over \gamma!}=\int_{0}^{\pi}sin^{\gamma}(\phi)d\phi$\\

for $\gamma$ odd, we obtain that;\\

$\int_{S(r)}(\Gamma''\times\Gamma')\centerdot d\overline{S}$\\

$=\sum_{w=0}^{\infty}\sum_{w'=0}^{\infty}\sum_{s\leq m,s,odd,s'\leq m,s',odd}(-1)^{s-1\over 2}C^{m}_{s}{(-1)^{w}m^{2w+1}\over (2w+1)!(r^{2}+1)^{w+{1\over 2}}}$\\

$(-1)^{s'-1\over 2}C^{m}_{s'}{(-1)^{w'}m^{2w'+1}\over (2w'+1)!(r^{2}+1)^{w'+{1\over 2}}}r^{2w+2w'+5}$\\

$(-\sum_{v=0,v,odd,v'=0,v',even}^{2w+1,2w'+1}C^{2w+1}_{v}C^{2w'+1}_{v'}I_{2w+2+m-s-v}^{s+v}I_{2w'+1+m-v'-s'}^{s'+v'+1}I_{2w+2w'+3-v-v'}^{v+v'+1}J_{2w+2w'+5}$\\

$-\sum_{v=0,v,odd,v'=0,v',odd}^{2w+1,2w'+1}C^{2w+1}_{v}C^{2w'+1}_{v'}I_{2w+2+m-s-v}^{s+v}I_{2w'+2-m-v'-s'}^{s'+v'}I_{2w+2w'+4-v-v'}^{v+v'}J_{2w+2w'+5}$\\

$-\sum_{v=0,v,even,v'=0,v',even}^{2w+1,2w'+1}C^{2w+1}_{v}C^{2w'+1}_{v'}I_{2w+1+m-s-v}^{s+v+1}I_{2w'+1+m-v'-s'}^{s'+v'+1}I_{2w+2w'+2-v-v'}^{v+v'+2}J_{2w+2w'+5}$\\

$-\sum_{v=0,v,even,v'=0,v',odd}^{2w+1,2w'+1}C^{2w+1}_{v}C^{2w'+1}_{v'}I_{2w+1+m-s-v}^{s+v+1}I_{2w'+2+m-v'-s'}^{s'+v'}I_{2w+2w'+3-v-v'}^{v+v'+1}J_{2w+2w'+5}$\\

$-\sum_{v=0,v,even,v'=0,v',even}^{2w+1,2w'+1}C^{2w+1}_{v}C^{2w'+1}_{v'}I_{2w+1+m-s-v}^{s+v+1}I_{2w'+1+m-v'-s'}^{s'+v'+1}I_{2w+2w'+2-v-v'}^{v+v'}(J_{2w+2w'+3}-J_{2w+2w'+5})$\\

$-\sum_{v=0,v,odd,v'=0,v',odd}^{2w+1,2w'+1}C^{2w+1}_{v}C^{2w'+1}_{v'}I_{2w+1++m-s-v}^{s+v+1}I_{2w'+2+m-v'-s'}^{s'+v'}I_{2w+2w'+2-v-v'}^{v+v'}(J_{2w+2w'+3}-J_{2w+2w'+5}))$\\

$=\sum_{w=0}^{\infty}\sum_{w'=0}^{\infty}\sum_{s\leq m,s,odd,s'\leq m,s',odd}(-1)^{s-1\over 2}C^{m}_{s}{(-1)^{w}m^{2w+1}\over (2w+1)!(r^{2}+1)^{w+{1\over 2}}}$\\

$(-1)^{s'-1\over 2}C^{m}_{s'}{(-1)^{w'}m^{2w'+1}\over (2w'+1)!(r^{2}+1)^{w'+{1\over 2}}}r^{2w+2w'+5}c_{w,w',s,s',m}$\\

where we have abbreviated the term in brackets to $c_{w,w',s,s',m}<0$.\\

We compute $\Gamma'''\times \Gamma''''$. By $(***)$ and $(*********)$;\\

$\Gamma'''=\sum_{w=0}^{\infty}\sum_{s\leq m,s, even}(-1)^{s\over 2}C^{m}_{s}{(-1)^{w}m^{2w+1}\over (2w+1)!(r^{2}+1)^{w+{1\over 2}}}$\\

$\sum_{v=0,v,even}^{2w+1}C^{2w+1}_{v}x^{2w+1-v}y^{v}(x,y,z)I_{2w+1+m-s-v}^{s+v}$\\

and;\\

$\Gamma''''=\sum_{w=0}^{\infty}\sum_{s\leq m,s, odd}(-1)^{s-1\over 2}C^{m}_{s}{(-1)^{w}m^{2w}\over (2w)!(r^{2}+1)^{w}}$\\

$(\sum_{v=0,v,odd}^{2w}C^{2w}_{v}x^{2w-v}y^{v}zI_{2w+1+m-v-s}^{s+v},\sum_{v=0,v,even}^{2w}C^{2w}_{v}x^{2w-v}y^{v}zI_{2w+m-v-s}^{s+v+1},$\\

$-\sum_{v=0,v,even}^{2w}C^{2w}_{v}x^{2w-v}y^{v+1}I_{2w+m-v-s}^{s+v+1}-\sum_{v=0,v,odd}^{2w}C^{2w}_{v}x^{2w+1-v}y^{v}I_{2w+1+m-v-s}^{s+v})$\\

so that;\\

$\Gamma'''\times\Gamma''''=\sum_{w=0}^{\infty}\sum_{w'=0}^{\infty}\sum_{s\leq m,s, even,s'\leq m,s',odd}(-1)^{s\over 2}C^{m}_{s}{(-1)^{w}m^{2w+1}\over (2w+1)!(r^{2}+1)^{w+{1\over 2}}}(-1)^{s'-1\over 2}C^{m}_{s'}{(-1)^{w'}m^{2w'}\over (2w')!(r^{2}+1)^{w'}}$\\

$(-\sum_{v=0,v,even,v'=0,v',even}^{2w+1,2w'}C^{2w+1}_{v}C^{2w'}_{v'}x^{2w+2w'+1-v-v'}y^{v+v'+2}I_{2w+1+m-s-v}^{s+v}I_{2w'+m-v'-s'}^{s'+v'+1}$\\

$-\sum_{v=0,v,even,v'=0,v',odd}^{2w+1,2w'}C^{2w+1}_{v}C^{2w'}_{v'}x^{2w+2w'+2-v-v'}y^{v+v'+1}I_{2w+1+m-s-v}^{s+v}I_{2w'+1+m-v'-s'}^{s'+v'}$\\

$-\sum_{v=0,v,even,v'=0,v',even}^{2w+1,2w'}C^{2w+1}_{v}C^{2w'}_{v'}x^{2w+2w'+1-v-v'}y^{v+v'}z^{2}I_{2w+1+m-s-v}^{s+v}I_{2w'+m-v'-s'}^{s'+v'+1},$\\

$+\sum_{v=0,v,even,v'=0,v',even}^{2w+1,2w'}C^{2w+1}_{v}C^{2w'}_{v'}x^{2w+2w'+2-v-v'}y^{v+v'+1}I_{2w+1+m-s-v}^{s+v}I_{2w'+m-v'-s'}^{s'+v'+1}$\\

$+\sum_{v=0,v,even,v'=0,v',odd}^{2w+1,2w'}C^{2w+1}_{v}C^{2w'}_{v'}x^{2w+2w'+3-v-v'}y^{v+v'}I_{2w+1+m-s-v}^{s+v}I_{2w'+1+m-v'-s'}^{s'+v'}$\\

$+\sum_{v=0,v,even,v'=0,v',odd}^{2w+1,2w'}C^{2w+1}_{v}C^{2w'}_{v'}x^{2w+2w'+1-v-v'}y^{v+v'}z^{2}I_{2w+1+m-s-v}^{s+v}I_{2w'+1+m-v'-s'}^{s'+v'},$\\

$+\sum_{v=0,v,even,v'=0,v',even}^{2w+1,2w'}C^{2w+1}_{v}C^{2w'}_{v'}x^{2w+2w'+2-v-v'}y^{v+v'}zI_{2w+1+m-s-v}^{s+v}I_{2w'+m-v'-s'}^{s'+v'+1}$\\

$-\sum_{v=0,v,even,v'=0,v',odd}^{2w+1,2w'}C^{2w+1}_{v}C^{2w'}_{v'}x^{2w+2w'+1-v-v'}y^{v+v'+1}zI_{2w+1+m-s-v}^{s+v}I_{2w'+1+m-v'-s'}^{s'+v'})$ $(\dag\dag\dag\dag\dag)$\\

Integrating $(\dag\dag\dag\dag\dag)$, we obtain;\\

$\int_{S(r)}(\Gamma'''\times\Gamma'''')\centerdot d\overline{S}$\\

$=\sum_{w=0}^{\infty}\sum_{w'=0}^{\infty}\sum_{s\leq m,s, even,s'\leq m,s',odd}(-1)^{s\over 2}C^{m}_{s}{(-1)^{w}m^{2w+1}\over (2w+1)!(r^{2}+1)^{w+{1\over 2}}}(-1)^{s'-1\over 2}C^{m}_{s'}{(-1)^{w'}m^{2w'}\over (2w')!(r^{2}+1)^{w'}}$\\

$(-\sum_{v=0,v,even,v'=0,v',even}^{2w+1,2w'}C^{2w+1}_{v}C^{2w'}_{v'}I_{2w+1+m-s-v}^{s+v}I_{2w'+m-v'-s'}^{s'+v'+1}$\\

$\int_{0}^{\pi}\int_{-\pi}^{\pi}r^{2w+2w'+5}sin^{2w+2w'+5}(\phi)cos^{2w+2w'+2-v-v'}(\theta)sin^{v+v'+2}(\theta)d\theta d\phi$\\

$-\sum_{v=0,v,even,v'=0,v',odd}^{2w+1,2w'}C^{2w+1}_{v}C^{2w'}_{v'}I_{2w+1+m-s-v}^{s+v}I_{2w'+1+m-v'-s'}^{s'+v'}$\\

$\int_{0}^{\pi}\int_{-\pi}^{\pi}r^{2w+2w'+5}sin^{2w+2w'+5}(\phi)cos^{2w+2w'+3-v-v'}(\theta)sin^{v+v'+1}(\theta)d\theta d\phi$\\

$-\sum_{v=0,v,even,v'=0,v',even}^{2w+1,2w'}C^{2w+1}_{v}C^{2w'}_{v'}I_{2w+1+m-s-v}^{s+v}I_{2w'+m-v'-s'}^{s'+v'+1}$\\

$\int_{0}^{\pi}\int_{-\pi}^{\pi}r^{2w+2w'+5}sin^{2w+2w'+3}cos^{2}(\phi)cos^{2w+2w'+2-v-v'}(\theta)sin^{v+v'}(\theta)d\theta d\phi$\\

$+\sum_{v=0,v,even,v'=0,v',even}^{2w+1,2w'}C^{2w+1}_{v}C^{2w'}_{v'}I_{2w+1+m-s-v}^{s+v}I_{2w'+m-v'-s'}^{s'+v'+1}$\\

$\int_{0}^{\pi}\int_{-\pi}^{\pi}r^{2w+2w'+5}sin^{2w+2w'+5}(\phi)cos^{2w+2w'+2-v-v'}(\theta)sin^{v+v'+2}(\theta)d\theta d\phi$\\

$+\sum_{v=0,v,even,v'=0,v',odd}^{2w+1,2w'}C^{2w+1}_{v}C^{2w'}_{v'}I_{2w+1+m-s-v}^{s+v}I_{2w'+1+m-v'-s'}^{s'+v'}$\\

$\int_{0}^{\pi}\int_{-\pi}^{\pi}r^{2w+2w'+5}sin^{2w+2w'+5}(\phi)cos^{2w+2w'+3-v-v'}(\theta)sin^{v+v'+1}(\theta)d\theta d\phi$\\

$+\sum_{v=0,v,even,v'=0,v',odd}^{2w+1,2w'}C^{2w+1}_{v}C^{2w'}_{v'}I_{2w+1+m-s-v}^{s+v}I_{2w'+1+m-v'-s'}^{s'+v'}$\\

$\int_{0}^{\pi}\int_{-\pi}^{\pi}r^{2w+2w'+5}sin^{2w+2w'+3}(\phi)cos^{2}(\phi)cos^{2w+2w'+1-v-v'}(\theta)sin^{v+v'+1}(\theta)d\theta d\phi$\\

$+\sum_{v=0,v,even,v'=0,v',even}^{2w+1,2w'}C^{2w+1}_{v}C^{2w'}_{v'}I_{2w+1+m-s-v}^{s+v}I_{2w'+m-v'-s'}^{s'+v'+1}$\\

$\int_{0}^{\pi}\int_{-\pi}^{\pi}r^{2w+2w'+5}sin^{2w+2w'+3}(\phi)cos^{2}(\phi)cos^{2w+2w'+2-v-v'}(\theta)sin^{v+v'}(\theta)d\theta d\phi$\\

$-\sum_{v=0,v,even,v'=0,v',odd}^{2w+1,2w'}C^{2w+1}_{v}C^{2w'}_{v'}I_{2w+1+m-s-v}^{s+v}I_{2w'+1+m-v'-s'}^{s'+v'}$\\

$\int_{0}^{\pi}\int_{-\pi}^{\pi}r^{2w+2w'+5}sin^{2w+2w'+3}(\phi)cos^{2}(\phi)cos^{2w+2w'+1-v-v'}(\theta)sin^{v+v'+1}(\theta)d\theta d\phi$\\

$(\dag\dag\dag\dag\dag\dag)$\\

An inspection of $(\dag\dag\dag\dag\dag\dag)$ shows that the terms 1 and 4, 2 and 5, 3 and 7, and 6 and 8 cancel again. This proves that;\\

$\int_{S(r)}(\Gamma'''\times\Gamma'''')\centerdot d\overline{S}=0$\\

By $(******)$, we have that;\\

$\Gamma'''''=\sum_{w=0}^{\infty}\sum_{s\leq m,s, odd}(-1)^{s-1\over 2}C^{m}_{s}{(-1)^{w}m^{2w}\over (2w)!(r^{2}+1)^{w}}$\\

$(-\sum_{v=0,v,even}^{2w}C^{2w}_{v}x^{2w-v}y^{v}I_{2w+m-v-s}^{v+s+1},\sum_{v=0,v,odd}^{2w}C^{2w}_{v}x^{2w-v}y^{v}I_{2w+1+m-v-s}^{v+s},0)$\\

It follows that;\\

$\Gamma'''''\times\Gamma''''$\\

$=\sum_{w=0}^{\infty}\sum_{w'=0}^{\infty}\sum_{s\leq m,s,odd,s'\leq m,s',odd}(-1)^{s-1\over 2}C^{m}_{s}{(-1)^{w}m^{2w}\over (2w)!(r^{2}+1)^{w}}(-1)^{s'-1\over 2}C^{m}_{s'}{(-1)^{w'}m^{2w'}\over (2w')!(r^{2}+1)^{w'}}$\\

$(-\sum_{v=0,v,odd,v'=0,v',even}^{2w,2w'}C^{2w}_{v}C^{2w'}_{v'}x^{2w+2w'-v-v'}y^{v+v'+1}I_{2w+1+m-s-v}^{s+v}I_{2w'+m-v'-s'}^{s'+v'+1}$\\

$-\sum_{v=0,v,odd,v'=0,v',odd}^{2w,2w'}C^{2w}_{v}C^{2w'}_{v'}x^{2w+2w'+1-v-v'}y^{v+v'}I_{2w+1+m-s-v}^{s+v}I_{2w'+1-m-v'-s'}^{s'+v'},$\\

$-\sum_{v=0,v,even,v'=0,v',even}^{2w,2w'}C^{2w}_{v}C^{2w'}_{v'}x^{2w+2w'-v-v'}y^{v+v'+1}I_{2w+m-s-v}^{s+v+1}I_{2w'+m-v'-s'}^{s'+v'+1}$\\

$-\sum_{v=0,v,even,v'=0,v',odd}^{2w,2w'}C^{2w}_{v}C^{2w'}_{v'}x^{2w+2w'+1-v-v'}y^{v+v'}I_{2w+m-s-v}^{s+v+1}I_{2w'+1+m-v'-s'}^{s'+v'},$\\

$-\sum_{v=0,v,even,v'=0,v',even}^{2w,2w'}C^{2w}_{v}C^{2w'}_{v'}x^{2w+2w'-v-v'}y^{v+v'}zI_{2w+m-s-v}^{s+v+1}I_{2w'+m-v'-s'}^{s'+v'+1}$\\

$-\sum_{v=0,v,odd,v'=0,v',odd}^{2w,2w'}C^{2w}_{v}C^{2w'}_{v'}x^{2w+2w'-v-v'}y^{v+v'}zI_{2w+1+m-s-v}^{s+v}I_{2w'+1+m-v'-s'}^{s'+v'})$ $(\dag\dag\dag\dag\dag\dag\dag)$\\

Integrating $(\dag\dag\dag\dag\dag\dag\dag)$, we get that;\\

$\int_{S(r)}(\Gamma'''''\times\Gamma'''')\centerdot d\overline{S}$\\

$=\sum_{w=0}^{\infty}\sum_{w'=0}^{\infty}\sum_{s\leq m,s,odd,s'\leq m,s',odd}(-1)^{s-1\over 2}C^{m}_{s}{(-1)^{w}m^{2w}\over (2w)!(r^{2}+1)^{w}}(-1)^{s'-1\over 2}C^{m}_{s'}{(-1)^{w'}m^{2w'}\over (2w')!(r^{2}+1)^{w'}}$\\

$(-\sum_{v=0,v,odd,v'=0,v',even}^{2w,2w'}C^{2w}_{v}C^{2w'}_{v'}I_{2w+1+m-s-v}^{s+v}I_{2w'+m-v'-s'}^{s'+v'+1}$\\

$\int_{0}^{\pi}\int_{-\pi}^{\pi}r^{2w+2w'+3}sin^{2w+2w'+3}(\phi)cos^{2w+2w'+1-v-v'}(\theta)sin^{v+v'+1}(\theta)d\theta d\phi$\\

$-\sum_{v=0,v,odd,v'=0,v',odd}^{2w,2w'}C^{2w}_{v}C^{2w'}_{v'}I_{2w+1+m-s-v}^{s+v}I_{2w'+1-m-v'-s'}^{s'+v'}$\\

$\int_{0}^{\pi}\int_{-\pi}^{\pi}r^{2w+2w'+3}sin^{2w+2w'+3}(\phi)cos^{2w+2w'+2-v-v'}(\theta)sin^{v+v'}(\theta)d\theta d\phi$\\

$-\sum_{v=0,v,even,v'=0,v',even}^{2w,2w'}C^{2w}_{v}C^{2w'}_{v'}I_{2w+m-s-v}^{s+v+1}I_{2w'+m-v'-s'}^{s'+v'+1}$\\

$\int_{0}^{\pi}\int_{-\pi}^{\pi}r^{2w+2w'+3}sin^{2w+2w'+3}(\phi)cos^{2w+2w'-v-v'}(\theta)sin^{v+v'+2}(\theta)d\theta d\phi$\\

$-\sum_{v=0,v,even,v'=0,v',odd}^{2w,2w'}C^{2w}_{v}C^{2w'}_{v'}I_{2w+m-s-v}^{s+v+1}I_{2w'+1+m-v'-s'}^{s'+v'}$\\

$\int_{0}^{\pi}\int_{-\pi}^{\pi}r^{2w+2w'+3}sin^{2w+2w'+3}(\phi)cos^{2w+2w'+1-v-v'}(\theta)sin^{v+v'+1}(\theta)d\theta d\phi$\\

$-\sum_{v=0,v,even,v'=0,v',even}^{2w,2w'}C^{2w}_{v}C^{2w'}_{v'}I_{2w+m-s-v}^{s+v+1}I_{2w'+m-v'-s'}^{s'+v'+1}$\\

$\int_{0}^{\pi}\int_{-\pi}^{\pi}r^{2w+2w'+3}sin^{2w+2w'+1}(\phi)cos^{2}(\phi)cos^{2w+2w'-v-v'}(\theta)sin^{v+v'}(\theta)d\theta d\phi$\\

$-\sum_{v=0,v,odd,v'=0,v',odd}^{2w,2w'}C^{2w}_{v}C^{2w'}_{v'}I_{2w+1+m-s-v}^{s+v}I_{2w'+1+m-v'-s'}^{s'+v'})$\\

$\int_{0}^{\pi}\int_{-\pi}^{\pi}r^{2w+2w'+3}sin^{2w+2w'+1}(\phi)cos^{2}(\phi)cos^{2w+2w'-v-v'}(\theta)sin^{v+v'}(\theta)d\theta d\phi$\\

$(\sharp)$\\

$=\sum_{w=0}^{\infty}\sum_{w'=0}^{\infty}\sum_{s\leq m,s,odd,s'\leq m,s',odd}(-1)^{s-1\over 2}C^{m}_{s}{(-1)^{w}m^{2w}\over (2w)!(r^{2}+1)^{w}}(-1)^{s'-1\over 2}C^{m}_{s'}{(-1)^{w'}m^{2w'}\over (2w')!(r^{2}+1)^{w'}}r^{2w+2w'+3}$\\

$(-\sum_{v=0,v,odd,v'=0,v',even}^{2w,2w'}C^{2w}_{v}C^{2w'}_{v'}I_{2w+1+m-s-v}^{s+v}I_{2w'+m-v'-s'}^{s'+v'+1}I_{2w+2w'+1-v-v'}^{v+v'+1}J_{2w+2w'+3}$\\

$-\sum_{v=0,v,odd,v'=0,v',odd}^{2w,2w'}C^{2w}_{v}C^{2w'}_{v'}I_{2w+1+m-s-v}^{s+v}I_{2w'+1-m-v'-s'}^{s'+v'}I_{2w+2w'+2-v-v'}^{v+v'}J_{2w+2w'+3}$\\

$-\sum_{v=0,v,even,v'=0,v',even}^{2w,2w'}C^{2w}_{v}C^{2w'}_{v'}I_{2w+m-s-v}^{s+v+1}I_{2w'+m-v'-s'}^{s'+v'+1}I_{2w+2w'-v-v'}^{v+v'+2}J_{2w+2w'+3}$\\

$-\sum_{v=0,v,even,v'=0,v',odd}^{2w,2w'}C^{2w}_{v}C^{2w'}_{v'}I_{2w+m-s-v}^{s+v+1}I_{2w'+1+m-v'-s'}^{s'+v'}I_{2w+2w'+1-v-v'}^{v+v'+1}J_{2w+2w'+3}$\\

$-\sum_{v=0,v,even,v'=0,v',even}^{2w,2w'}C^{2w}_{v}C^{2w'}_{v'}I_{2w+m-s-v}^{s+v+1}I_{2w'+m-v'-s'}^{s'+v'+1}I_{2w+2w'-v-v'}^{v+v'}(J_{2w+2w'+1}-J_{2w+2w'+3})$\\

$-\sum_{v=0,v,odd,v'=0,v',odd}^{2w,2w'}C^{2w}_{v}C^{2w'}_{v'}I_{2w+1+m-s-v}^{s+v}I_{2w'+1+m-v'-s'}^{s'+v'})I_{2w+2w'-v-v'}^{v+v'}(J_{2w+2w'+1}-J_{2w+2w'+3}))$\\

$=\sum_{w=0}^{\infty}\sum_{w'=0}^{\infty}\sum_{s\leq m,s,odd,s'\leq m,s',odd}(-1)^{s-1\over 2}C^{m}_{s}{(-1)^{w}m^{2w}\over (2w)!(r^{2}+1)^{w}}(-1)^{s'-1\over 2}$\\

$C^{m}_{s'}{(-1)^{w'}m^{2w'}\over (2w')!(r^{2}+1)^{w'}}r^{2w+2w'+3}d_{w,w',s,s',m}$\\

$(\sharp\sharp)$\\

where again we have denoted the term in brackets by $d_{w,w',s,s',m}$. We conclude that, if $m$ is even;\\

$P(r,t)=\int_{S(r)}(E_{2,e}\times B_{2,e})\centerdot d\overline{S}+\int_{S(r)}(E_{3,e}\times B_{2,e})\centerdot d\overline{S}$\\

$=\int_{S(r)}([\alpha\beta m^{2}[-sin^{2}(mt)cos^{2}(m{(r^{2}+1)^{1\over 2}\over c})-cos^{2}(mt)sin^{2}(m{(r^{2}+1)^{1\over 2}\over c})$\\

$+2sin(mt)cos(mt)sin(m{(r^{2}+1)^{1\over 2}\over c})cos(m{(r^{2}+1)^{1\over 2}\over c})]\Gamma\times\Gamma']+O({1\over r^{3}}))\centerdot d\overline{S}$\\

$+\int_{S(r)}(\beta\gamma m^{2}[-sin^{2}(mt)cos^{2}(m{(r^{2}+1)^{1\over 2}\over c})-cos^{2}(mt)sin^{2}(m{(r^{2}+1)^{1\over 2}\over c})$\\

$+2sin(mt)cos(mt)sin(m{(r^{2}+1)^{1\over 2}\over c})cos(m{(r^{2}+1)^{1\over 2}\over c})]\Gamma''\times\Gamma'+O({1\over r^{3}}))\centerdot d\overline{S}$\\

$=\beta\gamma m^{2}[-sin^{2}(mt)cos^{2}(m{(r^{2}+1)^{1\over 2}\over c})-cos^{2}(mt)sin^{2}(m{(r^{2}+1)^{1\over 2}\over c})$\\

$+2sin(mt)cos(mt)sin(m{(r^{2}+1)^{1\over 2}\over c})cos(m{(r^{2}+1)^{1\over 2}\over c})]$\\

$\sum_{w=0}^{\infty}\sum_{w'=0}^{\infty}\sum_{s\leq m,s,odd,s'\leq m,s',odd}(-1)^{s-1\over 2}C^{m}_{s}{(-1)^{w}m^{2w+1}\over (2w+1)!(r^{2}+1)^{w+{1\over 2}}}$\\

$(-1)^{s'-1\over 2}C^{m}_{s'}{(-1)^{w'}m^{2w'+1}\over (2w'+1)!(r^{2}+1)^{w'+{1\over 2}}}r^{2w+2w'+5}c_{w,w',s,s',m}+O({1\over r})$\\

Similarly, if $m$ is odd, we obtain that;\\

$P(r,t)=\int_{S(r)}(E_{2,o}\times B_{2,o})\centerdot d\overline{S}+\int_{S(r)}(E_{3,o}\times B_{2,o})\centerdot d\overline{S}$\\

$=\int_{S(r)}(E_{3,o}\times B_{2,o})\centerdot d\overline{S}$\\

$=\beta\gamma m^{2}[-cos^{2}(mt)cos^{2}(m{(r^{2}+1)^{1\over 2}\over c})-sin^{2}(mt)sin^{2}(m{(r^{2}+1)^{1\over 2}\over c})$\\

$-2sin(mt)cos(mt)sin(m{(r^{2}+1)^{1\over 2}\over c})cos(m{(r^{2}+1)^{1\over 2}\over c})]$\\

$\sum_{w=0}^{\infty}\sum_{w'=0}^{\infty}\sum_{s\leq m,s,odd,s'\leq m,s',odd}(-1)^{s-1\over 2}C^{m}_{s}{(-1)^{w}m^{2w}\over (2w)!(r^{2}+1)^{w}}(-1)^{s'-1\over 2}$\\

$C^{m}_{s'}{(-1)^{w'}m^{2w'}\over (2w')!(r^{2}+1)^{w'}}r^{2w+2w'+3}d_{w,w',s,s',m}+O({1\over r})$\\

\end{proof}

\begin{rmk}

The previous result shows that a single standing wave radiates, oscillating with time $t$ and radius $r$. We look for a cancellation by considering the other charge/current possibilities which satisfy the wave equation. This is the subject of the next theorem.

\end{rmk}

\begin{defn}
\label{configurations}
We let;\\

$\rho^{1}=cos(mx)cos(mt)$, $J^{1}=sin(mx)sin(mt)$\\

$\rho^{2}=cos(mx)sin(mt)$, $J^{2}=-sin(mx)cos(mt)$\\

$\rho^{3}=sin(mx)cos(mt)$, $J^{3}=-cos(mx)sin(mt)$\\

$\rho^{4}=sin(mx)sin(mt)$, $J^{4}=cos(mx)cos(mt)$\\

so that the corresponding pairs $(\rho_{i},\overline{J}_{i})$, for $1\leq i\leq 4$, satisfy the continuity equation, and satisfy the prescription of Lemma \ref{wave}.\\

We let $E^{i}_{k}$ and $B^{i}_{2}$, for $1\leq i\leq 4$, $2\leq k\leq 3$ be the corresponding causal fields.

\end{defn}

\begin{lemma}
\label{zero}
For $m$ even, we have that;\\

$\int_{S(r)}(E^{i}_{2}\times B^{j}_{2})\centerdot d\overline{S}=0$\\

for $1\leq i\leq j\leq 4$. Moreover;\\

$E^{1}_{3}=(\gamma mcos(mt)sin(m{(r^{2}+1)^{1\over 2}\over c})-\gamma msin(mt)cos(m{(r^{2}+1)^{1\over 2}\over c}))\Gamma''$\\

$B^{1}_{2}=(-\beta mcos(mt)sin(m{(r^{2}+1)^{1\over 2}\over c})+\beta msin(mt)cos(m{(r^{2}+1)^{1\over 2}\over c}))\Gamma'$\\

$E^{2}_{3}=(\gamma msin(mt)sin(m{(r^{2}+1)^{1\over 2}\over c})+\gamma mcos(mt)cos(m{(r^{2}+1)^{1\over 2}\over c}))\Gamma''$\\

$B^{2}_{2}=-(\beta msin(mt)sin(m{(r^{2}+1)^{1\over 2}\over c})+\beta mcos(mt)cos(m{(r^{2}+1)^{1\over 2}\over c}))\Gamma'$\\

$E^{3}_{3}=(-\gamma mcos(mt)sin(m{(r^{2}+1)^{1\over 2}\over c})+\gamma msin(mt)cos(m{(r^{2}+1)^{1\over 2}\over c}))\Delta''$\\

$B^{3}_{2}=(\beta mcos(mt)sin(m{(r^{2}+1)^{1\over 2}\over c})-\beta m sin(mt)cos(m{(r^{2}+1)^{1\over 2}\over c}))\Delta'$\\

$E^{4}_{3}=(-\gamma msin(mt)sin(m{(r^{2}+1)^{1\over 2}\over c})-\gamma mcos(mt)cos(m{(r^{2}+1)^{1\over 2}\over c}))\Delta''$\\

$B^{4}_{2}=(\beta msin(mt)sin(m{(r^{2}+1)^{1\over 2}\over c})+\beta mcos(mt)cos(m{(r^{2}+1)^{1\over 2}\over c}))\Delta'$\\

where;\\

$\Gamma'=\int_{-\pi}^{\pi}sin(m\theta)sin({mxcos(\theta)+mysin(\theta)\over (r^{2}+1)^{1\over 2}})(cos(\theta)z,sin(\theta)z,-sin(\theta)y-cos(\theta)x)d\theta$\\

$\Gamma''=\int_{-\pi}^{\pi}sin(m\theta)sin({mxcos(\theta)+mysin(\theta)\over (r^{2}+1)^{1\over 2}})(-sin(\theta),cos(\theta),0)d\theta$\\

$\Delta'=\int_{-\pi}^{\pi}cos(m\theta)sin({mxcos(\theta)+mysin(\theta)\over (r^{2}+1)^{1\over 2}})(cos(\theta)z,sin(\theta)z,-sin(\theta)y-cos(\theta)x)d\theta$\\

$\Delta''=\int_{-\pi}^{\pi}cos(m\theta)sin({mxcos(\theta)+mysin(\theta)\over (r^{2}+1)^{1\over 2}})(-sin(\theta),cos(\theta),0)d\theta$\\

\end{lemma}

\begin{proof}
It is easy to see, replacing summations over even indices, with odd indices, and vice versa, following the proof of the previous lemma, and assuming $m$ is even, that;\\

$E^{1}_{2}=c(r,m,t)\Gamma$, $E^{2}_{2}=d(r,m,t)\Gamma$\\

$E^{3}_{2}=e(r,m,t)\Delta$, $E^{4}_{2}=f(r,m,t)\Delta$\\

where;\\

$\Gamma=\int_{-\pi}^{\pi}cos(m\theta)cos({mxcos(\theta)+mysin(\theta)\over (r^{2}+1)^{1\over 2}})(x,y,z)d\theta$\\

$\Delta=\int_{-\pi}^{\pi}sin(m\theta)cos({mxcos(\theta)+mysin(\theta)\over (r^{2}+1)^{1\over 2}})(x,y,z)d\theta$\\

and $\{c,d,e,f\}$ are parameters not depending on $\theta,x,y$ or $z$. A similar argument works for $\Gamma'$ and $\Delta'$ as in the statement of the lemma. It is therefore sufficient to check that;\\

$\int_{S(r)}(\Gamma\times\Gamma')\centerdot d\overline{S}$\\

$=\int_{S(r)}(\Gamma\times\Delta')\centerdot d\overline{S}$\\

$=\int_{S(r)}(\Delta\times\Gamma')\centerdot d\overline{S}$\\

$=\int_{S(r)}(\Delta\times\Delta')\centerdot d\overline{S}=0$\\

The first case was checked in the previous lemma. The remaining cases follow from the first case, by replacing even with odd summations in both $v$ and $s$ when replacing $\Gamma$ by $\Delta$, and, similarly, for the pair $\Gamma'$ and $\Delta'$ in $v'$ and $s'$. For the remainder of the lemma, carefully follow the calculation in the previous result, the details are left to the reader.

\end{proof}
\begin{lemma}
\label{products}
For $m$ even, we have that;\\

$E_{3}^{1}\times B_{2}^{1}=C_{1}\Gamma''\times \Gamma'$\\

$E_{3}^{2}\times B_{2}^{2}=C_{2}\Gamma''\times \Gamma'$\\

$E_{3}^{2}\times B_{2}^{1}=-C_{3}\Gamma''\times \Gamma'$\\

$E_{3}^{1}\times B_{2}^{2}=-C_{3}\Gamma''\times \Gamma'$\\

$E_{3}^{3}\times B_{2}^{3}=C_{1}\Delta''\times \Delta'$\\

$E_{4}^{2}\times B_{2}^{3}=-C_{2}\Delta''\times \Delta'$\\

$E_{3}^{4}\times B_{2}^{3}=-C_{3}\Delta''\times \Delta'$\\

$E_{3}^{3}\times B_{2}^{4}=-C_{3}\Delta''\times \Delta'$\\

$E_{3}^{3}\times B_{2}^{1}=-C_{1}\Delta''\times \Gamma'$\\

$E_{3}^{4}\times B_{2}^{2}=-C_{2}\Delta''\times \Gamma'$\\

$E_{3}^{3}\times B_{2}^{2}=C_{3}\Delta''\times \Gamma'$\\

$E_{3}^{4}\times B_{2}^{1}=C_{3}\Delta''\times \Gamma'$\\

$E_{3}^{1}\times B_{2}^{3}=-C_{1}\Gamma''\times \Delta'$\\

$E_{3}^{2}\times B_{2}^{4}=-C_{2}\Gamma''\times \Delta'$\\

$E_{3}^{1}\times B_{2}^{4}=C_{3}\Gamma''\times \Delta'$\\

$E_{3}^{2}\times B_{2}^{3}=C_{3}\Gamma''\times \Delta'$\\

where;\\

 $C_{1}=\beta\gamma m^{2}(-cos^{2}(mt)sin^{2}(m{(r^{2}+1)^{1\over 2}\over c})+2sin(mt)cos(mt)sin(m{(r^{2}+1)^{1\over 2}\over c})cos(m{(r^{2}+1)^{1\over 2}\over c})-sin^{2}(mt)cos^{2}(m{(r^{2}+1)^{1\over 2}\over c}))$\\

 $C_{2}=\beta\gamma m^{2}(-sin^{2}(mt)sin^{2}(m{(r^{2}+1)^{1\over 2}\over c})-2sin(mt)cos(mt)sin(m{(r^{2}+1)^{1\over 2}\over c})cos(m{(r^{2}+1)^{1\over 2}\over c})-cos^{2}(mt)cos^{2}(m{(r^{2}+1)^{1\over 2}\over c}))$\\

 $C_{3}=\beta\gamma m^{2}(cos(mt)sin(mt)sin^{2}(m{(r^{2}+1)^{1\over 2}\over c})-sin(mt)cos(mt)cos^{2}(m{(r^{2}+1)^{1\over 2}\over c})-sin^{2}(mt)sin(m{(r^{2}+1)^{1\over 2}\over c})cos(m{(r^{2}+1)^{1\over 2}\over c})+cos^{2}(mt)sin(m{(r^{2}+1)^{1\over 2}\over c})cos(m{(r^{2}+1)^{1\over 2}\over c}))$\\

\end{lemma}

\begin{proof}
The proof is a simple calculation, using the result of the previous lemma.

\end{proof}

\begin{lemma}
\label{infinity}
There exists a family $(\rho,\overline{J})$ satisfying the continuity equation, with corresponding $(\rho,J)$ satisfying the wave equation, such that for the solution $(\rho,\overline{J},\overline{E},\overline{B})$ satisfying Maxwell's equations, obtained from Jefimenko's equations, for any $r>0$, $\int_{t}^{t+{\pi\over m}}P(r,t)dt=O({1\over r})$, where $\int_{t}^{t+{\pi\over m}}P(r,t)dt$ is the power radiated in a cycle, from a sphere $S(r)$ of radius $r$. In particularly, we have that;\\

$lim_{r\rightarrow\infty}\int_{t}^{t+{\pi\over m}}P(r,t)dt=0$\\

so the no radiation condition holds over a cycle.\\

The family is obtained by setting any three of $\{a_{1},a_{2},a_{3},a_{4}\}\subset\mathcal{R}$ to be equal, with the fourth having the reverse sign, and letting $\rho=a_{1}\rho_{1}+a_{2}\rho_{2}+a_{3}\rho_{3}+a_{4}\rho_{4}$, $J=a_{1}J_{1}+a_{2}J_{2}+a_{3}J_{3}+a_{4}J_{4}$.\\

\end{lemma}

\begin{proof}
We consider linear combinations $a_{1}\rho_{1}+a_{2}\rho_{2}+a_{3}\rho_{3}+a_{4}\rho_{4}$ and $a_{1}J_{1}+a_{2}J_{2}+a_{3}J_{3}+a_{4}J_{4}$ , where $\{a_{1},a_{2},a_{3},a_{4}\}$ are real scalars. Let $\{E_{m,comb},B_{m,comb}\}$ denote the resulting fields obtained from Jefimenko's equations. We have, by linearity, that;\\

$E_{2,m,comb}=\sum_{i=1}^{4}E_{2}^{i}$\\

$E_{3,m,comb}=\sum_{i=1}^{4}E_{3}^{i}$\\

$B_{2,m,comb}=\sum_{i=1}^{4}B_{2}^{i}$\\

and computing the power radiated through a sphere of radius $r$, using lemmas \ref{zero} and \ref{products};\\

$P(r,t)=\int_{S(r)}(E_{m,comb}\times B_{m,comb})\centerdot d\overline{S}+O({1\over r})$\\

$=\int_{S(r)}((E_{2,m,comb}+E_{3,m,comb})\times B_{2,m,comb})\centerdot d\overline{S}+O({1\over r})$\\

$=\int_{S(r)}(E_{2,m,comb}\times B_{2,m,comb})\centerdot d\overline{S}+\int_{S(r)}(E_{3,m,comb}\times B_{2,m,comb})\centerdot d\overline{S}+O({1\over r})$\\

$=\sum_{i,j=1}^{4}a_{i}a_{j}\int_{S(r)}(E_{2}^{i}\times B_{2}^{j})\centerdot d\overline{S}+\sum_{i=1}^{4}a_{i}a_{j}\int_{S(r)}(E_{3}^{i}\times B_{2}^{j})\centerdot d\overline{S}+O({1\over r})$\\

$=\sum_{i,j=1}^{4}a_{i}a_{j}\int_{S(r)}(E_{3}^{i}\times B_{2}^{j})\centerdot d\overline{S}+O({1\over r})$\\

$=(a_{1}^{2}C_{1}+a_{2}^{2}C_{2}-2a_{1}a_{2}C_{3})\int_{S(r)}(\Gamma''\times\Gamma')\centerdot d\overline{S}$\\

$+(a_{3}^{2}C_{1}+a_{4}^{2}C_{2}-2a_{3}a_{4}C_{3})\int_{S(r)}(\Delta''\times\Delta')\centerdot d\overline{S}$\\

$+(-a_{1}a_{3}C_{1}-a_{2}a_{4}C_{2}+(a_{2}a_{3}+a_{1}a_{4})C_{3})\int_{S(r)}(\Delta''\times\Gamma')\centerdot d\overline{S}$\\

$+(-a_{1}a_{3}C_{1}-a_{2}a_{4}C_{2}+(a_{1}a_{4}+a_{2}a_{3})C_{3})\int_{S(r)}(\Gamma''\times\Delta')\centerdot d\overline{S}+O({1\over r})$\\

We note the identities;\\

$C_{1}+C_{2}=\beta\gamma m^{2}(-cos^{2}(mt)sin^{2}(m{(r^{2}+1)^{1\over 2}\over c})+2sin(mt)cos(mt)sin(m{(r^{2}+1)^{1\over 2}\over c})cos(m{(r^{2}+1)^{1\over 2}\over c})-sin^{2}(mt)cos^{2}(m{(r^{2}+1)^{1\over 2}\over c}))$\\

$+\beta\gamma m^{2}(-sin^{2}(mt)sin^{2}(m{(r^{2}+1)^{1\over 2}\over c})-2sin(mt)cos(mt)sin(m{(r^{2}+1)^{1\over 2}\over c})cos(m{(r^{2}+1)^{1\over 2}\over c})-cos^{2}(mt)cos^{2}(m{(r^{2}+1)^{1\over 2}\over c}))$\\

$=\beta\gamma m^{2}(-sin^{2}(m{(r^{2}+1)^{1\over 2}\over c})-cos^{2}(m{(r^{2}+1)^{1\over 2}\over c}))$\\

$=-\beta\gamma m^{2}$\\

and;\\

$C_{1}-C_{2}=\beta\gamma m^{2}(-cos^{2}(mt)sin^{2}(m{(r^{2}+1)^{1\over 2}\over c})+2sin(mt)cos(mt)sin(m{(r^{2}+1)^{1\over 2}\over c})cos(m{(r^{2}+1)^{1\over 2}\over c})-sin^{2}(mt)cos^{2}(m{(r^{2}+1)^{1\over 2}\over c}))$\\

$-\beta\gamma m^{2}(-sin^{2}(mt)sin^{2}(m{(r^{2}+1)^{1\over 2}\over c})-2sin(mt)cos(mt)sin(m{(r^{2}+1)^{1\over 2}\over c})cos(m{(r^{2}+1)^{1\over 2}\over c})-cos^{2}(mt)cos^{2}(m{(r^{2}+1)^{1\over 2}\over c}))$\\

$=-\beta\gamma m^{2}(cos(2mt)sin^{2}(m{(r^{2}+1)^{1\over 2}\over c})+4sin(mt)cos(mt)sin(m{(r^{2}+1)^{1\over 2}\over c})cos(m{(r^{2}+1)^{1\over 2}\over c}))-cos(2mt)cos^{2}(m{(r^{2}+1)^{1\over 2}\over c}))$\\

$=\beta\gamma m^{2}(cos(2mt)(cos(2m{(r^{2}+1)^{1\over 2}\over c}))-sin(2mt)sin(2m{(r^{2}+1)^{1\over 2}\over c})$\\

$\int_{t}^{t+{\pi\over m}}C_{3}dt=\int_{t}^{t+{\pi\over m}}(\beta\gamma m^{2}(cos(mt)sin(mt)sin^{2}(m{(r^{2}+1)^{1\over 2}\over c})-sin(mt)cos(mt)cos^{2}(m{(r^{2}+1)^{1\over 2}\over c})-sin^{2}(mt)sin(m{(r^{2}+1)^{1\over 2}\over c})cos(m{(r^{2}+1)^{1\over 2}\over c})+cos^{2}(mt)sin(m{(r^{2}+1)^{1\over 2}\over c})cos(m{(r^{2}+1)^{1\over 2}\over c})))dt$\\

$=\int_{t}^{t+{\pi\over m}}(\beta\gamma m^{2}({sin(2mt)\over 2}sin^{2}(m{(r^{2}+1)^{1\over 2}\over c})-{sin(2mt)\over 2}cos^{2}(m{(r^{2}+1)^{1\over 2}\over c})$\\

$+cos(2mt)sin(m{(r^{2}+1)^{1\over 2}\over c})cos(m{(r^{2}+1)^{1\over 2}\over c})))dt$\\

$=0$\\

so we can obtain a convenient simplification by requiring that;\\

(i). $a_{1}^{2}=a_{2}^{2}$\\

(ii). $a_{3}^{2}=a_{4}^{2}$\\

(iii) $a_{1}a_{2}=-a_{3}a_{4}$\\

(iv). $a_{1}a_{3}=-a_{2}a_{4}$\\

(v). $a_{2}a_{3}=-a_{1}a_{4}$\\

The last $2$ conditions give that $a_{1}={-a_{2}a_{4}\over a_{3}}$ ($a_{3}\neq 0$)\\

$a_{2}a_{3}=-{-a_{2}a_{4}\over a_{3}}a_{4}$\\

$a_{3}^{2}=a_{4}^{2}$ ($a_{3}\neq 0$)\\

and $a_{3}={-a_{1}a_{4}\over a_{2}}$ ($a_{2}\neq 0$)\\

$a_{1}-{-a_{1}a_{4}\over a_{2}}=-a_{2}a_{4}$\\

$a_{1}^{2}=a_{2}^{2}$ ($a_{4}\neq 0$)\\

which are conditions $(i)$ and $(ii)$. Conversely, if $a_{1}=a_{2}$ and $a_{3}=-a_{4}$ or $a_{1}=-a_{2}$ and $a_{3}=a_{4}$, then conditions $(i),(ii),(iv),(v)$ are satisfied. Substituting into condition $(iii)$, we then require that;\\

$a_{1}^{2}=a_{3}^{2}$\\

which we can achieve if;\\

$a_{1}=a_{3}$ or $a_{1}=-a_{3}$\\

We then obtain that;\\

 $P(r,t)=-a_{1}^{2}\beta\gamma m^{2}\int_{S(r)}(\Gamma''\times\Gamma')\centerdot d\overline{S}-a_{3}^{2}\beta\gamma m^{2}\int_{S(r)}(\Delta''\times\Delta')\centerdot d\overline{S}$\\

$-2a_{1}a_{2}C_{3})\int_{S(r)}(\Gamma''\times\Gamma'-\Delta''\times\Delta')\centerdot d\overline{S}$\\

 $-a_{1}a_{3}\beta\gamma m^{2} cos(2mt)(cos(2m{(r^{2}+1)^{1\over 2}\over c}))-sin(2mt)sin(2m{(r^{2}+1)^{1\over 2}\over c})\int_{S(r)}(\Delta''\times\Gamma')\centerdot d\overline{S}$\\

 $-a_{1}a_{3}\beta\gamma m^{2}cos(2mt)(cos(2m{(r^{2}+1)^{1\over 2}\over c}))-sin(2mt)sin(2m{(r^{2}+1)^{1\over 2}\over c})\int_{S(r)}(\Gamma''\times\Delta')\centerdot d\overline{S}+O({1\over r})$\\

 If we integrate through a period of $\pi\over m$, we obtain that;\\

 $\int_{t}^{t+{\pi\over m}}P(r,t)dt$\\

 $=-a_{1}^{2}\beta\gamma m^{2}{\pi\over m}\int_{S(r)}(\Gamma''\times\Gamma'+\Delta''\times\Delta')\centerdot d\overline{S}+O({1\over r})$\\

 $=-\pi a_{1}^{2}\beta\gamma m\int_{S(r)}(\Gamma''\times\Gamma'+\Delta''\times\Delta')\centerdot d\overline{S}+O({1\over r})$\\

as $\int_{t}^{t+{\pi\over m}}cos(2mt)dt=\int_{t}^{t+{\pi\over m}}sin(2mt)dt=\int_{t}^{t+{\pi\over m}}C_{3}dt=0$\\

 By Poynting's Theorem, we have that;\\

 ${dW\over dt}=-{d\over dt}\int_{B(r)}{1\over 2}(\epsilon_{0}E^{2}+{1\over \mu_{0}}B^{2})dB(r)-{1\over \mu_{0}}P(r,t)$\\

 Integrating this expression over a period of ${\pi\over m}$ and using periodicity, from Jefimenko's equations that $\overline{E}(\overline{r},t+{\pi\over m})=-\overline{E}(\overline{r},t)$, and $\overline{B}(\overline{r},t+{\pi\over m})=-\overline{B}(\overline{r},t)$, we obtain that;\\

 $W(t+{\pi\over m})-W(t)=-\int_{B(r)}{1\over 2}(\epsilon_{0}E^{2}+{1\over \mu_{0}}B^{2})dB(r)|_{t+{\pi\over m}}+\int_{B(r)}{1\over 2}(\epsilon_{0}E^{2}+{1\over \mu_{0}}B^{2})dB(r)|_{t}-{1\over \mu_{0}}\int_{t}^{t+{\pi\over m}}P(r,t)$\\

 $={1\over \mu_{0}}\pi a_{1}^{2}\beta\gamma m\int_{S(r)}(\Gamma''\times\Gamma'+\Delta''\times\Delta')\centerdot d\overline{S}+O({1\over r})$\\

It follows that, for $r>0$;\\

$\int_{S(r)}(\Gamma''\times\Gamma'+\Delta''\times\Delta')\centerdot d\overline{S}$\\

$=\int_{B(r)}(\bigtriangledown\centerdot(\Gamma''\times\Gamma'+\Delta''\times\Delta'))dB(r)$\\

$={16\pi c^{3}c(t,m)\epsilon_{0}(r^{2}+1)^{3\over 2}\over a_{1}^{2}m}+O(r^{2})$\\

where $c(t,m)=W(t+{\pi\over m})-W(t)$, the difference in mechanical energy of the charge/current distribution confined to a vanishing annulus containing $S^{1}$, is independent of $r>1$, as contained in $B(r)$, (\footnote{\label{mechanical} This idea relies on an argument in \cite{G}, that the work done on a charge $q$ is $\overline{F}\centerdot d\overline{l}=q(\overline{E}+\overline{v}\times \overline{B})\centerdot \overline{v}dt=q\overline{E}\centerdot\overline{v}=\overline{E}\centerdot \overline{J}dt$, and can be summed over any $B(r)$ with $r>1$, as $\overline{J}$ is vanishing outside $S^{1}$. The idea that that $\overline{J}$ can be represented as $\rho\overline{v}$ is pursued in \cite{dep6}.}). Letting $f(x,y,z)=(\bigtriangledown\centerdot(\Gamma''\times\Gamma'+\Delta''\times\Delta'))$, we obtain that;\\

$\int_{B(r)}f dB(r)=d(t,m)(r^{2}+1)^{3\over 2}+O(r^{2})$\\

where $d(t,m)$ is independent of $r$. Dividing by $r^{3}$, we obtain that;\\

$lim_{r\rightarrow \infty}{1\over r^{3}}\int_{B(r)}f dB(r)=d(t,m)+O({1\over r})$\\

Taking a power series expansion of $f$ in the variables $\{x,y,z\}$, and integrating term by term, we can see that $f$ must be constant. Using the volume of the ball $B(r)$ as ${4\pi r^{3}\over 3}$, we must have that $f=d(t,m)=c(t,m)=0$. It follows, letting $r\rightarrow\infty$, that the mechanical energy of the matter wave doesn't vary over a cycle and moreover that the power radiated $\int_{t}^{t+m}P(r,t)dt$ over a ball $B(r)$ and a cycle is $O({1\over r})$, for any $r>0$, and $lim_{r\rightarrow\infty}\int_{t}^{t+{\pi\over m}}P(r,t)dt=0$. \\
\end{proof}

\begin{lemma}
\label{c}
The results of Lemma \ref{wave} and Lemma \ref{infinity} hold for the wave equation with velocity $c$;\\

${\partial \Psi\over \partial x^{2}}-{1\over c^{2}}{\partial \Psi\over \partial t^{2}}=0$, $(*)$\\

with the time cycle ${\pi\over m}$ replaced by ${\pi\over mc}$\\

\end{lemma}

\begin{proof}

The first result is clear. If we denote a solution $\Psi_{new}$ to $(*)$ by $\Psi(x,ct)$, and let $J_{new}=cJ(x,ct)$, with corresponding $\{\rho_{new},\overline{J}_{new}\}$, then in the Jefimenko equations, we get that;\\

$\overline{E}_{2,new}(x,t)=c\overline{E}_{2}[\rho_{new},{\overline{J}_{new}\over c}]$\\

$\overline{E}_{3,new}=c^{2}\overline{E}_{2}[\rho_{new},{\overline{J}_{new}\over c}]$\\

$\overline{B}_{2,new}=c^{2}\overline{B}_{2}[\rho_{new},{\overline{J}_{new}\over c}]$\\

Again, we obtain that;\\

$\int_{S(r)}(\overline{E}_{2,new}\times \overline{B}_{2,new})\centerdot d\overline{S}$\\

$=c^{3}\int_{S(r)}(\overline{E}_{2}\times \overline{B}_{2})[\rho_{new},{\overline{J}_{new}\over c}]\centerdot d\overline{S}$\\

$=0$\\

and;\\

$\int_{S(r)}(\overline{E}_{3,new}\times \overline{B}_{2,new})\centerdot d\overline{S}$\\

$=c^{4}\int_{S(r)}(\overline{E}_{3}\times \overline{B}_{2})[\rho_{new},{\overline{J}_{new}\over c}]\centerdot d\overline{S}$\\

so that, following Lemma \ref{2terms};\\

$P^{new}(r,t)+O({1\over r})=c^{4}P(r,t)[\rho_{new},{\overline{J}_{new}\over c}]+O({1\over r})$\\

We then obtain, for the linear combination in Lemma \ref{infinity} that;\\

$P^{new}(r,t)=c^{4}[-a_{1}^{2}\beta\gamma m^{2}\int_{S(r)}(\Gamma''\times\Gamma')\centerdot d\overline{S}-a_{3}^{2}\beta\gamma m^{2}\int_{S(r)}(\Delta''\times\Delta')\centerdot d\overline{S}$\\

$-2a_{1}a_{2}C_{3}'\int_{S(r)}(\Gamma''\times\Gamma'-\Delta''\times\Delta')\centerdot d\overline{S}$\\

 $-a_{1}a_{3}\beta\gamma m^{2}cos(2mct)(cos(2m{(r^{2}+1)^{1\over 2}\over c}))-sin(2mct)sin(2m{(r^{2}+1)^{1\over 2}\over c})\int_{S(r)}(\Delta''\times\Gamma')\centerdot d\overline{S}$\\

 $-a_{1}a_{3}\beta\gamma m^{2}cos(2mct)(cos(2m{(r^{2}+1)^{1\over 2}\over c}))-sin(2mct)sin(2m{(r^{2}+1)^{1\over 2}\over c})\int_{S(r)}(\Gamma''\times\Delta')\centerdot d\overline{S}]+O({1\over r})$\\

 where $C_{3}'=C_{3}(ct)$\\

 If we integrate through a period of $\pi\over mc$, we obtain that;\\

 $\int_{t}^{t+{\pi\over mc}}P^{new}(r,t)dt$\\

 $=-a_{1}^{2}\beta\gamma m^{2}{c^{4}\pi\over mc}\int_{S(r)}(\Gamma''\times\Gamma'+\Delta''\times\Delta')\centerdot d\overline{S}+O({1\over r})$\\

 $=-c^{3}\pi a_{1}^{2}\beta\gamma m\int_{S(r)}(\Gamma''\times\Gamma'+\Delta''\times\Delta')\centerdot d\overline{S}+O({1\over r})$\\

Now we can repeat the argument to obtain the same result, that $lim_{r\rightarrow\infty}\int_{t}^{t+{\pi\over mc}}P^{new}(r,t)dt=0$.

\end{proof}

\begin{lemma}
\label{constant}
Let $\rho_{0}=a$ and $J_{0}=b$ be constants, with $\{a,b\}\subset \mathcal{R}$, then for the corresponding $(\rho_{0},\overline{J}_{0})$ as in the paper, we have that ${\partial \rho_{0}\over \partial t}+\bigtriangledown.\centerdot\overline{J}_{0}=0$. If $\{\overline{E}_{0},\overline{B}_{0}\}$ are the corresponding causal fields, they satisfy the no radiation condition, and, if $(\rho,\overline{J})$ is the solution given either in Lemmas \ref{infinity} or \ref{c}, with $\{\overline{E}_{1},\overline{B}_{1}\}$ the corresponding causal fields, then they also satisfy the no radiation condition over a cycle, with the corresponding $(\rho+\rho_{0},\overline{J}+\overline{J}_{0})$ satisfying the continuity equation.

\end{lemma}
\begin{proof}
We clearly have that ${\partial \rho_{0}\over \partial t}+{\partial J_{0}\over \partial x}=0$, so that when we extend $J_{0}$ to $S^{1}$ by $\overline{J}_{0}(1,\theta)=J_{0}(\theta)(-sin(\theta),cos(\theta),0)$, for $-\pi\leq \theta<\pi$, we have for the further extension $(\rho_{0},\overline{J}_{0})$ on $Ann(1,\epsilon)\times (-\epsilon,\epsilon)$, using Lemma \ref{bump}, that ${\partial \rho_{0}\over \partial t}+\bigtriangledown\centerdot\overline{J}_{0}=0$. Using Lemma \ref{2terms}, we have that;\\

$lim_{r\rightarrow\infty}P(r)=lim_{r\rightarrow\infty}\int_{S(r)}(\overline{E}_{0,2}\times\overline{B}_{0,2}+\overline{E}_{0,3}\times\overline{B}_{0,2})d\overline{S}(r)$\\

but as $\dot{\rho_{0}}=0$ and $\dot{\overline{J}_{0}}=\overline{0}$, we have that $\overline{E}_{0,2}=\overline{E}_{0,3}=\overline{B}_{0,2}=\overline{0}$, so that $lim_{r\rightarrow\infty}P(r)=0$ and $\{\overline{E}_{0},\overline{B}_{0}\}$ satisfy the no radiating condition. Again, we have that ${\partial (\rho+\rho_{0})\over \partial t}+{\partial (J+J_{0})\over \partial x}=0$, with the same remark on the extension, so that ${\partial (\rho+\rho_{0})\over \partial t}+\bigtriangledown\centerdot(\overline{J}+\overline{J}_{0})=0$. We have that $\overline{E}_{1}=\overline{E}+\overline{E}_{0}$ and $\overline{B}_{1}=\overline{B}+\overline{B}_{0}$, so that;\\

$lim_{r\rightarrow\infty}P(r)=lim_{r\rightarrow\infty}\int_{S(r)}(\overline{E}_{1}\times\overline{B}_{1})d\overline{S}(r)$\\

$=lim_{r\rightarrow\infty}\int_{S(r)}((\overline{E}+\overline{E}_{0})\times(\overline{B}+\overline{B}_{0}))d\overline{S}(r)$\\

$=lim_{r\rightarrow\infty}(\int_{S(r)}(\overline{E}\times \overline{B})d\overline{S}(r)+\int_{S(r)}(\overline{E}_{0}\times \overline{B}_{0})d\overline{S}(r)$\\

$+\int_{S(r)}(\overline{E}\times \overline{B}_{0})d\overline{S}(r)+\int_{S(r)}(\overline{E}_{0}\times \overline{B})d\overline{S}(r))$\\

$=lim_{r\rightarrow\infty}(\int_{S(r)}(\overline{E}\times \overline{B}_{0})d\overline{S}(r)+\int_{S(r)}(\overline{E}_{0}\times \overline{B})d\overline{S}(r))$\\

$=lim_{r\rightarrow\infty}(\int_{S(r)}(\overline{E}_{2}\times \overline{B}_{0,2})d\overline{S}(r)+\int_{S(r)}(\overline{E}_{3}\times \overline{B}_{0,2})d\overline{S}(r))$\\

$+lim_{r\rightarrow\infty}(\int_{S(r)}(\overline{E}_{0,2}\times \overline{B}_{2})d\overline{S}(r)+\int_{S(r)}(\overline{E}_{0,3}\times \overline{B}_{2})d\overline{S}(r))$\\

$=0$\\

as required.\\

\end{proof}
\begin{defn}
\label{temperature}
If $(\rho,\overline{J})$ satisfy the continuity equation, let;\\

$T(\overline{x},t)=|{\overline{J}(\overline{x},t)\over \rho(\overline{x},t)}|$, if $\rho(\overline{x},t)\neq 0$\\

$T(\overline{x},t)=lim_{r\rightarrow 0}{1\over vol(B(\overline{x},r))}\int_{B(\overline{x},r)}|{\overline{J}(\overline{y},t)\over \rho(\overline{y},t)}|d\overline{y}$, if $\rho(\overline{x},t)=0$\\

We say that $(\rho,\overline{J})$ are in electromagnetic thermal equilibrium at time $t$, if $T(\overline{x},t)$ is constant on $\mathcal{R}^{3}$ and in electromagnetic thermal equilibrium if $T(\overline{x},t)$ is constant on $\mathcal{R}^{3}\times\mathcal{R}_{>0}$.\\

\end{defn}

\begin{rmk}
\label{boltzmann}
The definition is motivated by Boltzmann's definition of temperature for ideal gases as ${m|\overline{v}|^{2}\over 3k}$, see \cite{dep5}, where $|\overline{v}^{2}|$ is the mean square velocity of the particles making up the configuration, $m$ is the molecular mass and $k$ is Boltzmann's constant, and by the formula $\overline{J}=\rho\overline{v}$, see \cite{G} and \cite{dep6}. A local definition of temperature is given in \cite{dep5}, when the particles are moving with different velocities. In the electromagnetic case, if $\rho>0$, a natural definition of local average speed would be;\\

$V_{1}(\overline{x},t)=lim_{r\rightarrow 0}{\int_{B(\overline{x},r)}\rho(\overline{y},t)|\overline{v}(\overline{y},t)|dB(\overline{y})\over \int_{B(\overline{x},r)}\rho(\overline{y},t)dB(\overline{y})}$\\

$=lim_{r\rightarrow 0}{\int_{B(\overline{x},r)}|\overline{J}(\overline{y},t)|dB(\overline{y})\over \int_{B(\overline{x},r)}\rho(\overline{y},t)dB(\overline{y})}$\\

$=lim_{r\rightarrow 0}{{1\over vol(B(\overline{x},r))}\int_{B(\overline{x},r)}|\overline{J}(\overline{y},t)|dB(\overline{y})\over {1\over vol(B(\overline{x},r))}\int_{B(\overline{x},r)}\rho(\overline{y},t)dB(\overline{y})}$\\

$=|{\overline{J}(\overline{x},t)\over \rho(\overline{x},t)}|$\\

and a similar calculation works for the local average squared velocity;\\

$V_{2}(\overline{x},t)={|\overline{J}(\overline{x},t)|^{2}\over \rho^{2}(\overline{x},t)}$\\

We then establish electromagnetic thermal equilibrium when $V_{1}$ or $V_{2}$ is constant, and similarly for the generalisation $T$.
\end{rmk}

\begin{lemma}
\label{thermal}
If we set $a_{1}=-a_{4}$, $a_{2}=a_{3}$, then the system defined by $(\rho,\overline{J})$, for the linear combination $\rho=\sum_{i=1}^{4}\rho^{i}$, $J=\sum_{i=1}^{4}J^{i}$, from Definition \ref{configurations}, is in electromagnetic thermal equilibrium. In particularly, the configurations from Lemma \ref{infinity} are in electromagnetic thermal equilibrium for all $t>0$.

\end{lemma}

\begin{proof}

We have that;\\

$T(\overline{x},t)|_{S^{1}}={|\overline{J}(\overline{x},t)|\over |\rho(\overline{x},t)|}$\\

$=a|{a_{1}sin(mx)sin(mt)-a_{2}sin(mx)cos(mt)-a_{2}cos(mx)sin(mt)-a_{1}cos(mx)cos(mt)\over a_{1}cos(mx)cos(mt)+a_{2}cos(mx)sin(mt)+a_{2}sin(mx)cos(mt)-a_{1}sin(mx)sin(mt)}|$\\

$=|-1|$\\

$=1$\\

where $a=|(-sin(\theta),cos(\theta),0)|=1$, when $\rho(\overline{x},t)\neq 0$, so that, using Lemma \ref{bump}, for any configuration $(\rho,\overline{J})$, satisfying the continuity equation and extending $(\rho,J)$, we have  $T(\overline{x},t)|_{S^{1}}=1$.\\
\end{proof}
\end{section}

\end{document}